\documentclass[12pt,a4paper]{amsart}
\usepackage[utf8]{inputenc}
\usepackage{fancyhdr}
\usepackage{mathpazo}
\usepackage{setspace}
\usepackage[all]{xy}
\usepackage[active]{srcltx}
\usepackage{amsmath,amssymb}
\usepackage{mathtools}
\usepackage[a4paper, margin=2.5cm]{geometry}
\usepackage{longtable}
\usepackage{xcolor}
\usepackage{microtype}
\usepackage{amstext}
\usepackage{amsfonts}
\usepackage{amsmath}
\usepackage{graphicx}
\usepackage{afterpage,float,amsmath}
\usepackage{graphicx}
\usepackage{enumitem}
\usepackage{upgreek}
\usepackage{tikz}
\usepackage{multirow}
\usepackage{caption} 
\usepackage{mathrsfs}
\usepackage{amsthm}
\usepackage{hyperref} 
\usepackage{tikz-cd}
\usepackage{tkz-graph}
\usepackage{array}   
\usepackage{multirow} 
\usepackage{booktabs} 
\usepackage[normalem]{ulem}
\renewcommand{\S}{\mathcal{S}}

\newtheorem*{maintheorem}{Main Theorem}

\newtheorem{theorem}{Theorem}[section] 

\newtheorem{lemma}[theorem]{Lemma}      
\newtheorem{proposition}[theorem]{Proposition}

\theoremstyle{definition}
\newtheorem{definition}[theorem]{Definition} 
\theoremstyle{remark}
\usepackage{xcolor}

\newcommand{\citet}[2]{#1 \cite{#2}}
\newcommand{\M}{{\mathcal M}}
\newcommand{\E}{{\mathcal E}}
\newcommand{\F}{{\mathcal F}}
\newcommand{\N}{{\mathcal N}}
\newcommand{\Oo}{{\mathcal O}}

\newcommand{\R}{{\mathcal R}}

\renewcommand{\P}{\mathbb{P}}

\newcommand{\Pt}{\mathbb{P}^3}

\DeclareMathOperator{\Hom}{Hom}

\DeclareMathOperator{\coker}{coker}

\DeclareMathOperator{\Ext}{Ext}
\DeclareMathOperator{\rk}{{rk}}

\DeclareMathOperator{\opThree}{\mathcal{O}_{\mathbb{P}^3}}

\newcommand{\lhom}{{\mathcal H}{\it om}}
\newcommand{\lext}{{\mathcal E}{\it xt}}

\newcommand{\sing}{\operatorname{Sing}}

\newcommand{\onto}{\twoheadrightarrow}

\newcommand{\p}[1]{{\mathbb{P}^{#1}}}

\newcommand\restr[2]{{
  \left.\kern-\nulldelimiterspace 
  #1
  \vphantom{\big|} 
  \right|_{#2} 
  }}

\title[]{The moduli space of torsion-free sheaves with quasi-maximal third Chern class}

\author{Charles Almeida}
\address{ICEx - UFMG \\
Department of Mathematics,   Av. Ant\^onio Carlos, 6627\\
30123-970 Belo Horizonte, MG, Brazil}
\email{charlesalmeida@mat.ufmg.br}

\author{Marcos Jardim }
\address{IMECC - UNICAMP, Departamento de Matemática, Rua Sérgio Buarque de Holanda,
651, 13083-970 Campinas-SP, Brazil}
\email{jardim@ime.unicamp.br}

\author{Leonardo Oliveira}
\address{IMECC - UNICAMP, Departamento de Matemática, Rua Sérgio Buarque de Holanda,
651, 13083-970 Campinas-SP, Brazil}
\email{l237277@dac.unicamp.br}

\begin{document}

\begin{abstract}

In this paper, we investigate the geometry of the Gieseker moduli space of semistable rank $2$ torsion-free sheaves on $\mathbb{P}^3$ with Chern classes $(c_1, c_2, c_3) = (-1, c_2, c_2^2-2)$ for $c_2 \geq 2$.  We use the modular Serre correspondence to relate these moduli spaces to spaces of pairs and to suitable Hilbert schemes of one-dimensional subschemes in $\mathbb{P}^3$. For $c_2 \geq 4$, we prove that the moduli space is irreducible of dimension $c_2^2+3c_2+5$. For the case $c_2 = 3$, relying on a known geometric description of a relevant Hilbert scheme, we show that the moduli space $\M(-1,3,7)$ consists of exactly two irreducible components, namely the generic component of reflexive sheaves and a $T$-component, and we prove that their intersection is nonempty.

 	\medskip
   	 
   	 
   	 \medskip
   	 
\end{abstract}

\maketitle

\section{Introduction}

Let $\mathcal{M}(c_1,c_2,c_3)$ denote the Gieseker-Maruyama moduli scheme that parametrizes semistable rank $2$ sheaves on the projective space $\mathbb{P} ^3$, with $c_i$ denoting the $i$-th Chern class. We further consider the open subscheme  $\mathcal{R}(c_1,c_2,c_3)$ of $\mathcal{M}(c_1,c_2,c_3)$ corresponding to stable reflexive sheaves. 

Although the existence and projectivity of $\mathcal{M}(c_1,c_2,c_3)$ was proved by Maruyama, the geometry of these spaces remains largely unexplored. In contrast to the case of sheaves on $\mathbb{P} ^2$, where the moduli space is always irreducible, the spaces $\mathcal{M}(c_1,c_2,c_3)$ have a higher degree of complexity. In particular, the works \cite{jardim2017infinite} and \cite{almeida2022irreducible} show via concrete examples that these moduli spaces can have irreducible components of different dimensions.

 The Chern classes of semistable rank $2$ sheaves on $\mathbb{P}^3$ satisfy well-known restrictions. Up to normalization, we may assume $c_1 \in \{-1,0\}$, while $c_2 \geq 0$ by Bogomolov inequality and $c_1c_2 \equiv c_3 \pmod{2}$. In \cite{hartshorne1980reflexive}, R. Hartshorne proved that for fixed first and second Chern classes, there exists an upper bound for the third Chern class of stable rank $2$ reflexive sheaves on $\mathbb{P} ^3$. More precisely, Theorem 8.2 in \cite{hartshorne1980reflexive} shows that a necessary condition for $\mathcal{R}(c_1,c_2,c_3) \neq \varnothing$ is $c_2 \ge 0$ and $c_3 \leq c_{2} ^2 - c_2 +2$ if $c_1=0$, and $c_2>0$ and $c_3\leq c_2^2$ if $c_1=-1$. Furthermore, Hartshorne proved that  the moduli spaces $\mathcal{R}(c_1,c_2,c_{3, \max})$ are irreducible, smooth, and rational and Chang \cite{chang1983large} proved the same for $\mathcal{R}(0,c_2,c_{3, \max})$ for $c_2 \geq 4$. Okonek and Spindler \cite{Okonek1985Spektrum} proved that the bounds remain the same for torsion-free sheaves and computed the moduli spaces of sheaves with maximal $c_3$ for $c_2 \geq 6$, proving that these spaces are irreducible, smooth, rational, and projective varieties. B. Schmidt \cite{schmidt2020rank}, using different techniques, computed the moduli spaces of torsion-free sheaves with maximal $c_3$ for all $c_2$.

Below the maximal bound, the possible third Chern classes for reflexive sheaves on $\mathbb{P}^3$ admit significant gaps. In 1983, Chang \cite{chang1983large} proved that for given $c_2 > 0$ there are no stable rank $2$ reflexive sheaves on $\mathbb{P}^3$ with $c=(-1,c_2,c_3)$ where $$c_2^2-2c_2+4 <c_3 < c_2^2.$$ These gaps were subsequently generalized by Miró-Roig \cite{miro-roig1985gaps}. Specifically, defining $b(c_2) \coloneqq \frac{-1+\sqrt{4c_2-7}}{2}$, she proved that, for $c_2 \geq 4$ and $1 \leq r \leq b(c_2)$ there are no stable rank $2$ reflexive sheaves on $\mathbb{P}^3$ with $c_1=-1$ and $$c_2^2 -2r c_2+2(r+1)r < c_3 < c_2^2-2(r-1)c_2.$$

The authors of \cite{almeida2022irreducible} used the smoothness of certain moduli spaces of reflexive sheaves to produce new components for other moduli spaces, with lower third Chern classes. Indeed, for $s \geq 1$, by performing elementary transformations of sheaves in $\mathcal{R}(-1,c_2,c_2^2)$, along zero-dimensional subschemes of $\mathbb{P}^3$ of length $s$, it is possible to construct an irreducible component, denoted by $\operatorname{T}(-1,c_2,c_2^2,s)$, in $\M (-1,c_2,c_2^2-2s)$. The main goal of this paper is to prove that for $s=1$ and $c_2 \geq 4$, this is the only component of the moduli space. More precisely, we prove the following.

\begin{maintheorem}\label{maintheorem}
Let $\mathcal{M}(c_1,c_2,c_3)$ denote the Gieseker--Maruyama moduli space of rank 2 semistable torsion-free sheaves on $\P ^3$ with $i$-th Chern classes $c_i$; assume $c_1\in\{-1,0\}$. Then:

\begin{enumerate}
\item[(a)] $\M(c_1,c_2,c_3) \neq \varnothing$ if and only if
$c_1c_2 \equiv c_3 \pmod{2}$ and:

\begin{enumerate}
\item[(a1)] $c_1=c_2=0$, $c_3=-2n$ for $n\ge0$ and $n \neq 1$;

\item[(a2)] $c_1=0$, $c_2=1$, $c_3=-2n$ for $n\ge0$;

\item[(a3)] $c_1=0$, $c_2>1$ and $c_3\leq c_2^2-c_2+2$;

\item[(a4)] $c_1=-1$, $c_2>0$ and $c_3\leq c_2^2$.
\end{enumerate}

\item[(b)] $\M(-1,c,c^2-2)$ is irreducible, generically smooth and rational of dimension $c^2+3c+5$ when $c \geq 4$.

\item[(c)] $\M(-1,2,2)$ is connected and has exactly two generically smooth, rational irreducible components of dimension $11$ and $15$.

\item[(d)] $\M(-1,3,7)$ is connected and has exactly two irreducible generically smooth, rational components of dimension $19$ and $23$.
\end{enumerate}
\end{maintheorem}

Our approach relies on the modular Serre correspondence of \cite{jardim2025modular} in order to relate these moduli spaces to Hilbert schemes of curves in $\mathbb{P}^3$ whose geometry is better understood. In particular, this relationship explains why the gap found by Chang occurs in the first place.

This paper is organized as follows: Section \ref{sec:preliminaries} collects the preliminary results we will need; in particular, we briefly review Gieseker stability of torsion-free sheaves on $\P^3$ and the corresponding notion of stable coherent pairs, recording the form of the modular Serre correspondence of \cite{jardim2025modular}. Subsection \ref{sec:T-components} recalls the known families of irreducible components, in particular the T-components, described in \cite{almeida2022irreducible}, while Subsection \ref{nogap} proves part (a) of the Main Theorem. In Section \ref{sec:quasi-max} we translate the problem to the moduli space of pairs via the Serre Correspondence, and compute the relevant invariants of the corresponding  curves in $\P^3$, to establish irreducibility for the quasi-maximal case $c_3 = c_2^{\,2}-2$, proving part~(b) of the Main Theorem for $c_2 \geq 4$. Finally, in Sections \ref{sec:recover-2-2} and \ref{sec:M(-1,3,7)} we analyze the low degree boundary cases: Section \ref{sec:recover-2-2} covers item (c), which was originally established in \cite[Main Theorem 3(ii)]{almeida2022irreducible} and reproved here as a validation of our methods, while Section \ref{sec:M(-1,3,7)} shows that the moduli space $\M(-1,3,7)$ has exactly 2 irreducible components, completing the proof of part (d) of the Main Theorem.

Finally, we classify the non-reflexive sheaves in $\M(-1,2,2)$ and $\M(-1,3,7)$ in the last subsections of Sections~\ref{sec:recover-2-2} and~\ref{sec:M(-1,3,7)}, respectively, giving set-theoretic descriptions of the boundaries of their reflexive loci; see Propositions~\ref{prop:boundary-v2} and~\ref{prop:boundary-v3} for details.

\vspace{.4cm}
\noindent \textbf{Notation}:
We work over the field of complex numbers. We use calligraphic letters for sheaves, such as $\E$, $\F$, and the ideal sheaf $\mathcal{I}_Y$ of a subscheme $Y$. We write $[\E]$, $[\E,s]$, and $[Y]$ for the class represented by a sheaf, a coherent pair, and a subscheme in the corresponding moduli spaces or Hilbert schemes, respectively. When no confusion can arise, we suppress the brackets.  For coherent sheaves $\E$ and $\F$ and $i\geq0$, we set
\[
h^i(\E) \coloneqq \dim_{\mathbb{C}} H^i(\E),\qquad
\operatorname{ext}^i(\E,\F) \coloneqq \dim_{\mathbb{C}}\Ext^i(\E,\F).
\]

\vspace{.4cm}
\noindent \textbf{Acknowledgments}:
M.J. is partially supported by the CNPQ grant number 305601/2022-9 and the FAPESP-CEPID Project 2024/00923-6. C.A. is partially supported by CNPq Universal Project grant number 408974/2023-0. L. O. is supported by the Coordenação de Aperfeiçoamento de Pessoal de Nível Superior - Brasil (CAPES) - Finance Code 001. M.J. and L.O. would like to acknowledge support from the ICTP through the Associates Programme and from the Simons Foundation, whose grant (Record ID: SFI-MPS-T-Institutes-00012057, AD) also supported this work.

\vspace{.4cm}
\noindent \textbf{AI disclosure}:
AI were used for proofreading and improving the exposition. All mathematical content, arguments, and results are the authors’ own; the authors have independently verified the final manuscript and assume full responsibility for its content.

\section{Preliminaries}\label{sec:preliminaries}

\subsection{Stability of sheaves}\label{sec:stab-sheaves}

In this subsection, we recall the definitions of Gieseker semi-stability and $\mu$-semi-stability for sheaves. 

\begin{definition}
Fix an ample line bundle $\Oo _X (1)$ on $X$. The function $$P_{\E}(t) \coloneqq \chi (\E \otimes \Oo _X(t)) = \chi (\E (t))$$ is called the Hilbert polynomial of $\E$. Moreover, $$p_{\E} (t) \coloneqq \dfrac{P_{\E} (t)}{\rk (\E)}$$ is called the reduced Hilbert polynomial of $\E$.
\end{definition}

\begin{definition}
A torsion-free sheaf $\E$ is Gieseker semistable if for any nonzero proper subsheaf $\F \subset \E$ one has $p_{\F} \leq p_{\E}$. If the inequality is strict, $\E$ is called stable. 
\end{definition}

Let $X$ be a projective variety with a fixed ample line bundle $\Oo _X (1)$. For any torsion-free sheaf $\E$ on $X$, denote by $\alpha_i(\E)$ the$ i$-th coefficient of the Hilbert polynomial of $\E$. The number $\alpha_{\dim X -1} (\E) - \rk (\E) \cdot \alpha _{\dim X -1} (\Oo _X)$ is called the degree of the sheaf $\E$ and we will denote it by $\deg (\E)$.

\begin{definition}
Let $\E$ be a torsion-free sheaf on a projective variety $X$, with a fixed line bundle $\Oo _X(1)$. We define the slope of $\E$ as the number: $$\mu (\E) \coloneqq \dfrac{\deg (\E)}{\rk (\E)}.$$ We say that $\E$ is $\mu$-semistable if, for all subsheaves $\F \subset \E $ with $0 < \rk (\F) < \rk (\E)$, we have $\mu (\F) \leq \mu (\E)$. If the last inequality is strict, we say that $\E$ is stable. 
\end{definition}

We have the following chain of implications of the notions of stability. 

\begin{proposition}
Let $\E$ be a torsion-free sheaf on a projective variety $X$ with a fixed ample line bundle $\Oo _X (1)$. Then we have the following chain of implications.

\begin{center}
$\E$ is $\mu$-stable $\Longrightarrow$ $\E$ is Gieseker stable $\Longrightarrow$ $\E$ is Gieseker semistable $\Longrightarrow$$\E$ is $\mu$-semistable.
\end{center}
\end{proposition}

\begin{proof}
See \cite[Lemma 1.2.13]{huybrechts2010geometry}.
\end{proof}

\begin{proposition}
Let $\E$ be a rank $2$ torsion-free sheaf on $\P ^3$ with $\deg (\E)=-1$. If $\E$ is $\mu$-semistable, then $\E$ is $\mu$-stable.
\end{proposition}

\begin{proof}
See \cite[Lemma 1.2.14]{huybrechts2010geometry}.
\end{proof}

\subsection{Stability of pairs}\label{sec:stab-pairs}

In \cite{jardim2025modular}, the notion of stability of pairs is used to provide a systematic construction of the Serre correspondence at the level of morphisms of schemes, connecting the moduli space of sheaves with the Hilbert scheme of subschemes of codimension $2$ in $\mathbb{P}^3$. We will use this correspondence to study the moduli spaces  $\M(-1,c,c^2-2)$. For this, we start by recalling some definitions. 

\begin{definition}
A coherent pair $(\E, s)$ on $X$ consists of a coherent sheaf $\E$ and a (possibly zero) global section $s \in H^0 (\E)$. We say that $(\E,s)$ is pure if $\E$ is a pure sheaf and $s$ is a non-trivial section.
\end{definition}

A morphism $(\phi, \lambda) \colon (\F, s') \rightarrow (\E,s)$ between coherent pairs is a pair $(\phi, \lambda)$, where $\phi \in \Hom (\F,\E)$ and $\lambda \in \mathbb{C}$, such that $\phi \circ s' = \lambda \cdot s$, i.e., the following diagram commutes

\[\begin{tikzcd}
	{\mathcal{O} _X} & {\mathcal{O} _X} \\
	{\mathcal{F}} & {\mathcal{E}.}
	\arrow["\lambda", from=1-1, to=1-2]
	\arrow["{s'}"', from=1-1, to=2-1]
	\arrow["s", from=1-2, to=2-2]
	\arrow["\phi"', from=2-1, to=2-2]
\end{tikzcd}\]

\begin{definition}
A pure coherent pair $(\mathcal{E},s)$ is called saturated if the cokernel of the induced monomorphism $s \colon \mathcal{O}_X \hookrightarrow \mathcal{E}$ is torsion-free. 
\end{definition}

The construction of the moduli space of coherent pairs relies on a notion of stability for pairs that is related to Gieseker stability.

\begin{definition}
Let $\delta \in \mathbb{Q}[t]_{>0}$ be a rational polynomial with a positive leading coefficient. For a given coherent pair $(\mathcal{E},s)$, define its reduced Hilbert polynomial as

$$p^{\delta} _{(\mathcal{E},s)} (t) \coloneqq \dfrac{P_{\mathcal{E}}(t)+\varepsilon \cdot \delta}{\rk (\mathcal{E})},$$
where
$$\varepsilon = \left \{ \begin{matrix} 1, & \mbox{if } s \neq 0, \\ 0, & \mbox{if }s=0 .\end{matrix} \right. $$
A pair $(\mathcal{E},s)$ is $\delta$-(semi)-stable if it is pure and every proper, nontrivial sub-pair $(\mathcal{F},s') \hookrightarrow(\mathcal{E},s)$ satisfies $$p^{\delta}_{(\mathcal{F},s')} (t) < (\leq) p^{\delta}_{(\mathcal{E},s)}(t),$$ where polynomials are ordered lexicographically.

\end{definition}
With the definition of stability of pairs, we can consider the moduli space $\mathcal{S}^{\delta}(v)$ of $\delta$-semistable pairs with class $v$, constructed in \cite{wandel2015moduli}.

If the stability parameter $\delta$ is sufficiently small, then $\delta$-stability implies Gieseker stability, regardless of the section, as the next result shows. 

\begin{proposition}\label{verystable}
Let $\mathcal{E}$ be a torsion-free sheaf of rank $r$ with $h^0(\mathcal{E}) > 0$, and fix $\delta < 1/(r-1)$. If a pair $(\mathcal{E},s)$ is $\delta$-stable, then $\mathcal{E}$ is semistable. 
\end{proposition}

\begin{proof}
See \cite[Proposition 2.4]{jardim2025modular}.
\end{proof}

For a sufficiently small perturbation of the Hilbert polynomial, the converse holds provided the section is non-zero.

\begin{proposition}\label{prop:stable-sheaf-stable-pair}
Let $\mathcal{E}$ be a torsion-free sheaf of rank $r$ with $h^0(\mathcal{E}) > 0$, and fix $\delta < 1/ r$. If $\mathcal{E}$ is stable, then for every non-trivial section $s \in H^0(\mathcal{E})$, the pair $(\mathcal{E},s)$ is $\delta$-stable. 
\end{proposition}

\begin{proof}
See \cite[Proposition 2.5]{jardim2025modular}.
\end{proof}

We are interested in cases where both the sheaf and the pair are stable.

\begin{definition}
For a fixed $\delta \in \mathbb{Q}[t]_{>0}$, a coherent pair $(\mathcal{E},s)$ is called very stable if $(\mathcal{E},s)$ is $\delta$-stable and $\mathcal{E}$ is Gieseker stable. 
\end{definition}

From now on, we fix $X=\P^3$. If $\E$ is a rank $2$ sheaf, we define 
$$ v(\E) = (c_1(\E),c_2(\E),c_3(\E)) \in \mathbb{Z}^3 $$
where $c_i$ is the $i$-th Chern class. Conversely, the i-th component of a vector $v\in\mathbb{Z}^3$ is denoted by $c_i(v)$. Moreover, if $v=v(\E)$, then we also define the twisted class $v(k)\coloneqq v(\E(k))$ where $k \in \mathbb{Z}$.

The next two theorems describe the morphisms between the moduli space of pairs and the Gieseker moduli space and a morphism from the space of pairs to a Hilbert scheme; let $\operatorname{Hilb}^{d,g}$ denote the Hilbert scheme of subschemes $C\subset\P^3$ with Hilbert polynomial $dt+(1-g)$.

\begin{theorem}\label{ModularSerreHilbert}
Fix a numerical class $v$.
For every $\delta$ such that every $\delta$-semistable pair $(\E,s)$ is saturated, there exists a morphism 
$$\begin{matrix}
    \Gamma \colon & \mathcal{S}^{\delta}(v) & \rightarrow & \operatorname{Hilb}^{d,g} \\
& [\E,s] & \mapsto & [C]
\end{matrix}$$
from the moduli space $\mathcal{S}^{\delta}(v)$ to the Hilbert scheme $\operatorname{Hilb}^{d,g}$ where 
$$ d=c_2(v) ~~{\rm and} ~~ g=\dfrac{1}{2}(c_2(v)c_1(v)+c_3(v)) - 2c_2(v) +1.$$
Here $C\subset\P^3$ is the subscheme determined by $\coker(s)\simeq\mathcal{I}_C(c_1(v))$. The fibers $\Gamma^{-1}([C])$ are isomorphic to $\P\Ext^1(\mathcal{I}_C,\Oo_{\P^3}(-c_1(v)))$.
\end{theorem}
\begin{proof}
This is a particular case of \cite[Main Theorem 1]{jardim2025modular}. If the pair $(\E,s)$ is saturated, then $\coker\{s \colon \mathcal{O}_{\P^3}\to\E\}$ is a torsion-free sheaf of rank 1 whose determinant is equal to $\det(\E)$. Therefore, we have the short exact sequence
$$ 0 \to \mathcal{O}_{\P^3} \stackrel{s}{\to} \E \to \mathcal{I}_C(c_1(v)) \to 0 ,$$
where $C\subset\P^3$ is a closed subscheme. Comparing second Chern classes in the displayed sequence gives $\deg(C)=c_2(\E)$. Taking the Euler characteristic $\chi$, we obtain:
\begin{equation}\label{eq:chi}
 \chi(\E) = 1 + \chi(\mathcal{I}_C(c_1)) = 1+\chi(\mathcal{O}_{\mathbb{P}^3}(c_1)) - \chi(\mathcal{O}_C(c_1))   
\end{equation}
By the Hirzebruch--Riemann--Roch theorem on $C$,
\[
\chi(\mathcal{O}_C(c_1)) = \deg(\mathcal{O}_C(c_1)) + 1 - g = c_1 d + 1 - g
\]
On the other hand, we get
$$ \chi(\E) - \chi(\mathcal{O}_{\mathbb{P}^3}(c_1)) = 1 - 2c_2 - \frac{1}{2}c_1 c_2 + \frac{1}{2}c_3 $$
Substituting the last two expressions in the equality in display \eqref{eq:chi} and using $d=c_2$, we obtain
$$ 1 - 2c_2 - \frac{1}{2}c_1 c_2 + \frac{1}{2}c_3 = g - c_1 c_2 $$
which we can solve for $g$ and obtain the stated expression.
\end{proof}

We can add a technical lemma, which will be useful in this work.

\begin{lemma}\label{CorGammaProper}
Fix a numerical class $v$. For any $\delta$ such that every $\delta$-semistable pair $(\E,s)$ is saturated, the morphism $$\Gamma \colon \mathcal{S}^{\delta}(v) \rightarrow \operatorname{Hilb}^{d,g}$$ is proper.
\end{lemma}

\begin{proof}
The moduli space $\mathcal{S}^{\delta}(v)$ is constructed as a GIT quotient (see \cite{wandel2015moduli} for more details), which implies that it is projective (and, in particular, proper) over the base field $\mathbb{C}$. Similarly, the Hilbert scheme $\operatorname{Hilb}^{d,g}$ is well-known to be projective over $\mathbb{C}$, \cite[Theorem 3.1]{Grothendieck1961Techniques}, and is therefore a separated scheme of finite type over $\mathbb{C}$.

Consider the commutative diagram formed by the morphism $\Gamma$ and the respective structure morphisms to $\operatorname{Spec}(\mathbb{C})$:

\[\begin{tikzcd}
	{\mathcal{S}^{\delta}(v)} && {\operatorname{Hilb}^{d,g}} \\
	& {\operatorname{Spec} (\mathbb{C})}
	\arrow["\Gamma", from=1-1, to=1-3]
	\arrow[from=1-1, to=2-2]
	\arrow[from=1-3, to=2-2]
\end{tikzcd}\]

Because the structure morphism $\mathcal{S}^{\delta}(v) \rightarrow \operatorname{Spec}({\mathbb{C}})$ is proper and the target structure morphism $\operatorname{Hilb}^{d,g} \rightarrow \operatorname{Spec}(\mathbb{C})$ is separated, it follows from the cancellation property for proper morphisms (e.g., Stacks Project, Tag 01W6) that $\Gamma$ itself must be a proper morphism.
\end{proof}

\begin{theorem}\label{ModularSerreGieseker}
Fix a numerical class $v$. For every $\delta$ such that every $\delta$-semistable pair $(\E,s)$ is very stable, there exists a morphism
$$\begin{matrix}
\Psi \colon & \mathcal{S}^{\delta}(v) & \rightarrow & \mathcal{M}(v) \\
& [\E,s] & \mapsto & [\E] 
\end{matrix}$$
from the moduli space of pairs $\mathcal{S}^{\delta}(v)$ to the Gieseker moduli space $\mathcal{M}(v)$. In addition, if $[\E] \in \M(v)$, then $\Psi^{-1}([\E])=\P H^0(\E)$.

\end{theorem}

\begin{proof}
See \cite[Main Theorem 1]{jardim2025modular}.
\end{proof}

For the remainder of this paper, we fix the following notation.

\begin{definition}\label{def:chernclass}
For any integer $c \geq 2$, we denote by $v_c \coloneqq (-1, c, c^2-2)$ and $v_c(1)\coloneqq(1,c,c^2-2)$. 
\end{definition}

Our strategy is to find a single stability parameter $\delta$ for which the moduli space $\mathcal{S}^{\delta}(v)$ admits well-defined morphisms to both the Gieseker moduli space of sheaves and the corresponding Hilbert scheme.

\subsection{Known families}\label{sec:T-components}

We now recall a construction of an infinite series of irreducible components of the moduli space of rank $2$ torsion-free sheaves on $\mathbb{P}^3$. To state those results, we will need the following piece of notation:
$$ \Sigma \coloneqq \big\{ (-1,c,c^2) ~|~ c\ge1 \big\} \cup \big\{  (0,c,c^2-c+2) ~|~ c\ge 3 \big\} \cup \{(0,1,0)\}. $$

Key facts about the geometry of the moduli space $\R(e,n,m)$ of stable rank 2 reflexive sheaves with Chern classes $(c_1,c_2,c_3)=(e,n,m)$ are summarized in our next result.

\begin{theorem}\label{maximalreflexive}
For each triple $(e, n,m)\in\Sigma$, $\R(e,n,m)$ is a fine moduli space which is a dense open subset of $\M(e,n,m)$. In addition, $\R(e,n,m)$ is an irreducible, non-singular, and rational scheme whose dimension is given by
\begin{itemize}
\item[(a)] $\dim\R(-1,1,1)=3$ and $\dim\R(-1,c,c^2)=c^2+3c+1$ for $c\ge2$.
\item[(b)] $\dim\R(0,c,c^2-c+2)$ equals $8c-3$ when $c=3 $, and $c^2+2c+5$ when $c\ge4$. Moreover, $\dim\R(0,1,0)=5$.
\end{itemize}
\end{theorem}

\begin{proof}
Item \textit{(a)} is precisely \cite[Theorem 9.2]{hartshorne1980reflexive}. Item \textit{(b)} is a consequence of \cite[Theorem 1.1]{schmidt2020rank} as $\R(0,c,c^2-c+2)$ and $\R(0,1,0)$ are open subsets of $\M(0,c,c^2-c+2)$ and $\M(0,1,0)$, respectively. 
\end{proof}

We can then apply \cite[Theorem 9 and Theorem 13]{almeida2022irreducible} to obtain an infinite series of irreducible components, which we call \textit{T-components}, of $\M(e,n,m-2s)$ where $s>0$.

\begin{theorem}\label{0dcomp}
For each $(e,n,m)\in\Sigma\cup\{(0,2,4)\}$ and every $s>0$, there is a generically reduced irreducible component $\operatorname{T}(e,n,m,s)\subset\M(e,n,m-2s)$ of dimension $\dim\R(e,n,m)+4s$ whose generic point $[\E]\in\operatorname{T}(e,n,m,s)$ satisfies 
$$ [\E^{\vee\vee}]\in\R(e,n,m) ~~{\rm and} ~~ \E^{\vee\vee}/\E\simeq\Oo_S, $$
where $S$ is a reduced 0-dimensional scheme of length $h^0(\Oo_S)=s$. In addition, $\operatorname{T}(e,n,m,s)$ is rational when $(e,n,m)\in\Sigma$.
\end{theorem} 

\begin{proof}
This statement is part of \cite[Theorem 13]{almeida2022irreducible}, since, except for the triples $(0,2,4)$ and $(0,1,0)$, the set $\Sigma$ defined above is a subset of the set $\Sigma_0\cup\Sigma_{-1}$ defined in \cite[displays (41) and (42)]{almeida2022irreducible}. We observe that the moduli space $\R(0,1,0)$ also satisfies the properties (I) through (V) outlined in \cite[pages 24-25]{almeida2022irreducible}, as described in Theorem \ref{maximalreflexive}.

Moreover, the hypothesis $0 \leq 2s \leq m$ in \cite[Theorem 13]{almeida2022irreducible} amounts to $c_3 \geq 0$ and reflects the goal of that paper, which was to compactify the moduli spaces of reflexive and locally free sheaves. No step in its proof uses this bound, and every $s > 0$ provides an irreducible component $\operatorname{T}(e,n,m,s)$.
The range $c_3<0$ is of independent interest, for instance due to its connection to Quot schemes, see \cite{almeida2026quot}.

The case $(e,n,m)=(0,2,4)$ is exceptional because $\mathcal{R}(0,2,4)$ is not a fine moduli space; in this case, we can use \cite[Theorem 9]{almeida2022irreducible} to guarantee the existence of the irreducible components $\operatorname{T}(0,2,4,s)$, although they are not rational.
\end{proof}

As the preceding theorem shows, the main ingredients for constructing new irreducible components of the moduli space of torsion-free sheaves on $\mathbb{P}^3$ are the smooth irreducible components of the moduli space of reflexive sheaves.

\begin{theorem}\label{Xscheme}
Let $e,n,m,r,s$ be integers such that $e \in \{ -1,0\}$, $n$, $m >0$, $en \equiv m \pmod 2$, $r \geq e$ and $s \geq0$. Then the following statements hold.

\begin{enumerate}
\item The scheme $\overline{\operatorname{X}}(e,n,m,r,s)$ is equidimensional of dimension $8n+4s+2r+2+e$, and the set of irreducible components of $\overline{\operatorname{X}}(e,n,m,r,s)$ is in bijection with that of $\overline{\mathcal{R}}^0(e,n,m)\coloneqq\{[\F]\in\R(e,n,m)\mid\Ext^2(\F,\F)=0\}$. Under this bijection, to any irreducible component $\mathcal{R}^*$ of $\overline{\mathcal{R}}^0(e,n,m)$ there corresponds an irreducible component $\operatorname{X}(e,n,m,r,s)$ of $\overline{\operatorname{X}}(e,n,m,r,s)$ specified by the property that, for a generic sheaf $[\E]$ of $\operatorname{X}(e,n,m,r,s)$, $[\E^{\vee \vee}] \in \mathcal{R}^*$ and $\mathcal{Q}_{\E} \coloneqq \E^{\vee \vee}/\E$ is a sheaf of the form $\mathcal{Q}_{\E} = \Oo _L(r) \oplus \Oo_S$, where $L$ is a line and $S$ is a reduced 0-dimensional scheme of length $h^0(\Oo_S)=s$.

\item If $0 \geq r \geq e$, then for a generic sheaf $[\E] \in \operatorname{X}(e,n,m,r,s)$, $$\dim \operatorname{Ext}^1(\E,\E)=8n+4s+5+2e > \dim \operatorname{X}(e,n,m,r,s).$$
\end{enumerate}
\end{theorem}
\begin{proof} See \cite[Theorem 10]{almeida2022irreducible}. \end{proof}

\subsection{Proof of the Main Theorem, item (a)} \label{nogap}

We start by proving the ``only if" claim. The necessity of the condition $c_1c_2 \equiv c_3 \pmod{2}$ is established in \cite[Corollary 2.4]{hartshorne1980reflexive}. The Bogomolov inequality implies that $\M(c_1,c_2,c_3)=\emptyset$ for $c_1=0$ and $c_2<0$, and for $c_1=-1$ and $c_2<1$. Finally, \cite[Theorem 1.1]{schmidt2020rank} proves that $\M(c_1,c_2,c_3)=\emptyset$ for $c_1=0$, and $c_3>c_2^2-c_2+2$ when $c_2>1$ and $c_3>0$ when $c_2\in\{0,1\}$, and for $c_1=-1$ and $c_3>c_2^2$. Moreover, $\mathcal{M}(0,0,-2)=\emptyset$ by \cite[Main Theorem]{guimaraes2025moduli} \footnote{In the notation of this reference $\N_{\P^3}(2,n)$ actually corresponds to $\M(0,0,-2n)$ in the notation of the present paper.}.

Conversely, we must argue that $\M(c_1,c_2,c_3)\ne\emptyset$ when the numerical conditions are satisfied.

Regarding item \textit{(a1)}, the non emptiness of $\mathcal{M}(0,0,c_3)$ when $c_3\le-4$ was proved in \cite[Main Theorem]{guimaraes2025moduli}. In fact, the authors proved that $\mathcal{M}(0,0,c_3)$ is isomorphic to $\operatorname{Sym}^2(\P^3)$ for $c_3=-4$ and it admits an irreducible component of dimension $-2c_3-3$ for $c_3 \leq -6$.

For the remaining items in part (a), we invoke Theorem \ref{0dcomp}, noting that the triples $(c_1,c_2,c_3)\in\mathbb{Z}^3$ not covered above are precisely those in the set 
\begin{gather*}
\{(-1,n,m) ~|~ n>0, ~m\equiv n  \pmod{2} , ~m\le n^2 \} \cup \\
\cup\{(0,n,m) ~|~ n>1, m\le n^2-n+2, m\in2\mathbb{Z}\} \cup \{(0,1,-2s) ~|~ s\ge0 \}.
\end{gather*} 
The triples in $\Sigma$ itself correspond to nonempty moduli spaces by Theorem~\ref{maximalreflexive}, while every other triple in the set above is of the form $(e,n,m-2s)$ with $(e,n,m)\in\Sigma\cup\{(0,2,4)\}$ and $s>0$. Besides concluding that $\M(e,n,m) \neq \emptyset$ when the triple $(e,n,m)$ belongs to the set above, we also have that for each $s>0$:
\begin{itemize}
\item[(a)] $\M(-1,c,c^2-2s)$ contains a generically reduced, rational irreducible component $\operatorname{T}(-1,c,c^2,s)$ of dimension
$$ \dim  \operatorname{T}(-1,c,c^2,s) = \begin{cases} 
    3+4s, & \text{if } c=1 \\
    c^2+3c+1+4s, & \text{if } c>1. 
\end{cases}$$
\item[(b)] $\M(0,c,c^2-c+2-2s)$ for $c\ge2$ contains a generically reduced, rational irreducible component $\operatorname{T}(0,c,c^2-c+2,s)$ of dimension
$$ \dim  \operatorname{T}(0,c,c^2-c+2,s) = \begin{cases} 
    8c-3+4s, & \text{if } c=2,3, \\
    c^2+2c+5+4s, & \text{if } c>3. 
\end{cases}$$
\item[(c)] $\M(0,2,4-2s)$ contains a generically reduced irreducible component $\operatorname{T}(0,2,4,s)$ of dimension $13+4s$.
\item[(d)] $\M(0,1,-2s)$ contains a generically reduced, rational irreducible component $\operatorname{T}(0,1,0,s)$ of dimension $5+4s$.
\end{itemize}

\section{Sheaves with quasi-maximal \texorpdfstring{$c_3$}{c3}}\label{sec:quasi-max}

Chang proved in \cite{chang1983large} that for a given $c_2>0$ there are no stable rank $2$ reflexive sheaves on $\P^3$ with $c_1=-1$ and $c_2^2-2c_2+4 < c_3 < c_2^2$; Miró-Roig later improved this in \cite[Theorem A]{miro-roig1985gaps}, showing that further gaps in the possible values of $c_3$ occur as $c_2$ increases.

When $c_2\geq4$, the value $c_3=c_2^2-2$ strictly falls within Chang's gap, meaning that no stable rank 2 reflexive sheaves with Chern classes $(c_1,c_2,c_3)=(-1,c_2,c_2^2-2)$ exist. However, by Theorem \ref{0dcomp}, torsion-free sheaves do exist and form the T-component $\operatorname{T}(-1,c_2,c_2^2,1)$. We will prove that $\M(-1,c_2,c_2^2-2)$ is, in fact, irreducible when $c_2\geq 4$, consisting solely of the T-component.

\subsection{Existence of global section}\label{sec:global-section}

To study the moduli spaces we are interested in, we will need a non-zero section of $\E(1)$ for every sheaf $\E$ in them. The relevant technical tool to establish the existence of such a section is the notion of the \textit{spectrum} of a torsion-free sheaf, which we now recall; the interested reader can check \cite{okonek1984spektrum} and \cite{almeida2020spectrum}. For completeness, we include some definitions and results here.

\begin{theorem}\label{thm-spectrum}
Let $\mathcal{E}$ be a rank $r$ torsion-free sheaf on $\mathbb{P}^3$, with generic splitting type $(a_1,\dots,a_r)$, with $a_i \in \mathbb{Z}$, $a_1 \leq a_2 \leq \cdots \leq a_r$, and $s_{\E}=h^0(\mathcal{E}xt^2(\mathcal{E}, \mathcal{O}_{\mathbb{P}^3}))$. Then there exists a list of $m$ integers $(k_1,k_2,\dots,k_m)$ with $k_1 \leq k_2 \leq \dots \leq k_m$ such that 

\begin{enumerate}
\item[a)] $h^1(\mathcal{E}(l)) = s_{\E} + \displaystyle \sum _{i=1}^m h^0(\mathcal{O}_{\mathbb{P}^1} (k_i +l +1))$, if $l \leq -a_r-1$;

\item[b)] $h^2(\mathcal{E}(l)) =  \displaystyle \sum _{i=1}^m h^1(\mathcal{O}_{\mathbb{P}^1} (k_i +l +1))$, if $l \geq -a_1-3$.
\end{enumerate}
\end{theorem}

\begin{proof}
See \cite[Theorem 1]{almeida2020spectrum}.
\end{proof}

\begin{definition}
Let $\E$ be a rank $r$ torsion-free sheaf on $\P ^3$, with generic splitting type $(a_1,\dots,a_r)$, with $a_i \in \mathbb{Z}$, and $a_1 \leq a_2 \leq \cdots \leq a_r$, and $s _{\E} = h^0( \mathcal{E}\operatorname{xt}^2(\E, \Oo _{\P ^3}))$. Then the list of $m$ integers $(k_1,k_2,\dots,k_m)$, provided by the previous theorem, is called the spectrum of the sheaf $\E$.
\end{definition}

We will now restrict our attention to stable rank 2 sheaves. The following result provides the conditions for an integer to belong to the spectrum of such a sheaf.

\begin{proposition}\label{charspec}
Let $\E$ be a normalized stable rank $2$ torsion-free sheaf on $\P^3$.
Let $m$ be the length of its spectrum, and let $H\subset\P^3$ be a
general hyperplane, chosen so that
$H\cap\operatorname{Supp}(\lext^2(\E,\Oo_{\P^3}))=\varnothing$.
Set
\[
s_{\E_H} \coloneqq h^0\bigl(\lext^1(\E,\Oo_{\P^3})\otimes\Oo_H\bigr).
\]
Then
\[
m=-\chi(\E_H(-1))=c_2(\E).
\]

And if $c_1(\E) =-1$, then $$c_3(\E) = -2 \sum k_i - c_2(\E) -2s_{\E}.$$

Moreover, the spectrum $(k_1,\dots,k_m)$ of $\E$ satisfies:
\begin{enumerate}
\item[a)] Let $k>1$. If at least $s_{\E_H}+1$ entries $k_i$ satisfy
$k_i\geq k$, then every integer $k'$ with $1\leq k'\leq k$ appears
in the spectrum.
\item[b)] If $k\leq c_1(\E)-1$ appears in the spectrum, then every
integer $k'$ with $k\leq k'\leq-1$ appears in the spectrum.
\end{enumerate}
\end{proposition}

\begin{proof}
The equality $m=c_2(\E)$ is \cite[Proposition 6]{almeida2020spectrum} and the equality $$c_3(\E) = -2 \sum k_i - c_2(\E) -2s_{\E}$$ is \cite[Proposition 7]{almeida2020spectrum}. The Euler-characteristic equality printed in the former has a sign error:
Riemann--Roch on a general plane gives
\[
\chi(\E_H(-1))=\frac{c_1(\E)(c_1(\E)+1)}{2}-c_2(\E)
=c_2(\E),
\]
since $c_1(\E)\in\{-1,0\}$. In particular, for $c_1(\E)=-1$,
$\chi(\E_H(t))=(t+1)^2-c_2(\E)$.
The Grauert--M\"ulich restriction theorem gives generic splitting type
$(c_1(\E),0)$, namely $(-1,0)$ or $(0,0)$.
The two connectedness assertions now follow from
\cite[Proposition 8(a),(b)]{almeida2020spectrum}.
\end{proof}

For later use, we record two consequences of Proposition~\ref{charspec} for
a stable rank $2$ torsion-free sheaf $\E$ with $c_1(\E)=-1$. Rearranging the formula for $c_3(\E)$, the spectrum $(k_1,\dots,k_{c_2})$, which has length $c_2=c_2(\E)$ by Proposition~\ref{charspec}, satisfies
\begin{equation}\label{eq:spectrum-sum-minus-one}
\sum_{i=1}^{c_2}k_i=-\frac{c_2(\E)+c_3(\E)}{2}-s_\E ,
\end{equation}
and item (b) implies that every entry is at least $-c_2(\E)$, since an entry
$k\leq-2$ forces the $|k|$ distinct integers $k,\dots,-1$ to occur. Moreover,
Riemann--Roch gives
\begin{equation}\label{eq:RR-spectrum}
\chi(\E(t))=\frac{(t+1)(t+2)(2t+3)}{6}-\frac{c_2(\E)(2t+3)}{2}
+\frac{c_3(\E)}{2}.
\end{equation}

We can now bound $h^2(\E(1))$ in terms of the spectrum, which is what gives
the existence of sections.

\begin{lemma}\label{section}
Let $\E$ be a stable rank $2$ torsion-free sheaf on $\P^3$, with
$c_1=-1$ and $c_2\geq2$. Then
\[
h^2(\E(1))\leq\binom{c_2-2}{2}.
\]
For $c_2\geq4$, equality holds if and only if the spectrum is
$(-c_2,-c_2+1,\dots,-1)$. For $c_2=2,3$, both sides vanish for every
such sheaf. In particular,
\[
h^0(\E(1))\geq\chi(\E(1))-h^2(\E(1))
\geq\frac{c_3-c_2^2+4}{2}.
\]
Thus $h^0(\E(1))\neq0$ whenever $c_3\geq c_2^2-2$.
\end{lemma}

\begin{proof}
By Proposition~\ref{charspec}, the spectrum has length $c_2$.
Negative connectedness implies $k_i\geq-c_2$ for every entry:
an entry smaller than $-c_2$ would force more than $c_2$ distinct
integers to appear. By Theorem~\ref{thm-spectrum}(b),
\[
h^2(\E(1))=\sum_{i=1}^{c_2}\max\{0,-k_i-3\}.
\]
If $c_2=2$ or $3$, every summand is zero, so
$h^2(\E(1))=0=\binom{c_2-2}{2}$, independently of the spectrum.

Assume now $c_2\geq4$ and put $N_j=\#\{i:k_i\leq-j\}$.
Then $N_j=0$ for $j\geq c_2+1$ and
\[
h^2(\E(1))=\sum_{j=4}^{c_2}N_j.
\]
If $N_j>0$, negative connectedness forces the $j-1$ distinct entries
$-j+1,\dots,-1$. Hence $N_j\leq c_2-j+1$, and
\[
h^2(\E(1))\leq\sum_{j=4}^{c_2}(c_2-j+1)
=\frac{(c_2-2)(c_2-3)}{2}=\binom{c_2-2}{2}.
\]
Equality holds if and only if $N_j=c_2-j+1$ for every
$j=4,\dots,c_2$. In particular, equality implies $N_{c_2}=1$,
so $-c_2$ occurs. Connectedness and the length of the spectrum then
force $(-c_2,-c_2+1,\dots,-1)$, with each entry occurring once.
Conversely, this spectrum attains the bound.

Finally, stability and Serre duality give $h^3(\E(1))=0$, so
\[
h^0(\E(1))\geq h^0(\E(1))-h^1(\E(1))
=\chi(\E(1))-h^2(\E(1)).
\]
Riemann--Roch gives $\chi(\E(1))=(c_3-5c_2+10)/2$.
Using the bound, valid also for $c_2=2,3$, we obtain
\[
h^0(\E(1))\geq\frac{c_3-5c_2+10}{2}
-\frac{(c_2-2)(c_2-3)}{2}
=\frac{c_3-c_2^2+4}{2}.
\]
This is at least $1$ when $c_3\geq c_2^2-2$.
\end{proof}

\subsection{Pairs with quasi-maximal \texorpdfstring{$c_3$}{c3}}\label{sec:pairs-quasi-max}

If $\E$ is a rank 2 torsion-free sheaf with Chern classes $(c_1,c_2,c_3)=(-1,c,c^2-2)$, then by the previous result we will have $h^0(\E(1))>0$. Using this fact, we can show that pairs of the form $(\E(1),s)$ are always saturated.

\begin{lemma}\label{idealsheaf}
Let $\F$ be a rank 2 reflexive sheaf on $\P^3$ with %
$h^0(\F)=0$. If $h^0(\F(1))> 0$, then for every nonzero $s \in H^0(\F(1))$, the sheaf 
$\mathcal{Q} \coloneqq \coker\{s \colon \Oo_{\P^3} \hookrightarrow \F(1)\}$ is a rank 1 torsion-free sheaf.   
\end{lemma}
\begin{proof}
By the additivity of the rank, $\rk \mathcal{Q} = \rk \F(1) - \rk \Oo_{\P ^3} = 1$. So it remains to prove that $\mathcal{Q}$ is torsion-free. %

If $\mathcal{Q}$ is not torsion-free, then let $\mathcal{T} \hookrightarrow \mathcal{Q}$ be its maximal torsion subsheaf, so that $\mathcal{Q}/\mathcal{T}$ is a torsion-free sheaf of rank 1.
Consider the composed epimorphism ${\F(1)} \twoheadrightarrow \mathcal{Q} \twoheadrightarrow \mathcal{Q}/\mathcal{T}$, and let $\mathcal{G}$ be its kernel. Since the kernel of a map from a reflexive sheaf to a torsion-free sheaf is reflexive and reflexive sheaves with rank $1$ are locally free, we have that $\mathcal{G} \simeq \Oo_{\P ^3} (k)$. We therefore have the following commutative diagram:
$$\xymatrix{
    & & \mathcal{T} \ar@{^{(}->}[d] \\
    \mathcal{O}_{\P^3} \ar@{^{(}->}[r]^-{s} \ar[d] & \mathcal{F}(1) \ar@{->>}[r] \ar@{=}[d] & \mathcal{Q} \ar[d] \\
    \mathcal{O}_{\P^3}(k) \ar@{^{(}->}[r] & \mathcal{F}(1) \ar@{->>}[r] & \mathcal{Q}/\mathcal{T}
}$$

The left-hand column gives a monomorphism
$\Oo_{\P^3}\hookrightarrow\Oo_{\P^3}(k)$, so $k\geq0$. On the other hand,
twisting the monomorphism $\Oo_{\P^3}(k)\hookrightarrow\F(1)$ by
$\Oo_{\P^3}(-1)$ gives $\Oo_{\P^3}(k-1)\hookrightarrow\F$; if $k\geq1$,
this would force $h^0(\F)\geq h^0(\Oo_{\P^3}(k-1))>0$, contradicting the
hypothesis. Hence $k=0$, the monomorphism $\Oo_{\P^3}\to\Oo_{\P^3}(k)$ is a
nonzero endomorphism of $\Oo_{\P^3}$ and therefore an isomorphism, and the
snake lemma applied to the diagram gives
\[
\mathcal T\simeq\coker\{\Oo_{\P^3}\to\Oo_{\P^3}\}=0 .
\]
Thus $\mathcal Q$ is torsion-free.
\end{proof}

\begin{lemma}\label{lem:quotienttorsionfree}
Let $\mathcal{E}$ be a normalized, $\mu$-stable rank 2 torsion-free sheaf.
If $h^0(\mathcal{E}(1)) \neq 0$, then for every non-zero section $s \in H^0(\mathcal{E}(1))$, the cokernel of the induced monomorphism $\mathcal{Q} \coloneqq \coker (s \colon \mathcal{O}_{\P ^3} \hookrightarrow \mathcal{E}(1))$ is torsion-free.
\end{lemma}

\begin{proof}
Let $s \colon \mathcal{O}_{\mathbb{P}^3} \hookrightarrow \mathcal{E}(1)$ be a non-zero section, and let $\mathcal{Q}$ be its cokernel. Consider the standard short exact sequence relating $\mathcal{E}$ to its reflexive hull $\mathcal{E}^{\vee\vee}$, twisted by $\mathcal{O}_{\mathbb{P}^3}(1)$:
$$0 \longrightarrow \mathcal{E}(1) \longrightarrow \mathcal{E}^{\vee\vee}(1) \longrightarrow \mathcal{Q}_{\mathcal{E}} (1)\longrightarrow 0,$$
Since $\mathcal{E}^{\vee\vee}$ is a normalized $\mu$-stable rank $2$ reflexive sheaf, we have that $h^0(\E^{\vee\vee})=0$.
We can apply Lemma \ref{idealsheaf} to any section $s\in H^0(\E(1))\subset H^0(\E^{\vee\vee}(1))$
to obtain the following exact commutative diagram: 
\[\begin{tikzcd}
	{\mathcal{O}_{\P ^3}} & {\mathcal{O}_{\P ^3}} & \\
	{\mathcal{E}(1)} & {\mathcal{E}^{\vee \vee}(1)} & {\mathcal{Q}_{\mathcal{E}}(1)} \\
	{\mathcal{Q}} & {\mathcal{I}_C}(c_1+2) & {\mathcal{Q}_{\mathcal{E}}(1)}
	\arrow[shift right, no head, from=1-1, to=1-2]
	\arrow[no head, from=1-1, to=1-2]
	\arrow[hook, from=1-1, to=2-1]
	\arrow[hook, from=1-2, to=2-2]
	\arrow[hook, from=2-1, to=2-2]
	\arrow[two heads, from=2-1, to=3-1]
	\arrow[two heads, from=2-2, to=2-3]
	\arrow[two heads, from=2-2, to=3-2]
	\arrow[shift right, no head, from=2-3, to=3-3]
	\arrow[shift left, no head, from=2-3, to=3-3]
	\arrow[hook, from=3-1, to=3-2]
	\arrow[two heads, from=3-2, to=3-3]
\end{tikzcd}\]
From the third row, $\mathcal{Q}$ is a subsheaf of the torsion-free sheaf $\mathcal{I}_C(c_1+2)$, which implies that $\mathcal{Q}$ is torsion-free of rank 1. More precisely, $\mathcal{Q}\simeq\mathcal{I}_{C'}(c_1+2)$ for some 1-dimensional scheme $C'$ containing $C$.
\end{proof}

We summarize the results obtained in this subsection in the following statement. 

\begin{theorem}\label{nowall}
Let $v_c$ be as in Definition \ref{def:chernclass} and fix $0 < \delta < \frac{1}{2}$. Then the natural morphisms
\begin{center}
 \begin{tikzcd}[row sep=large, column sep=large]
    & \mathcal{S}^{\delta}(v_c(1)) \arrow[dl, "\Psi"'] \arrow[dr, "\Gamma"] & \\
    \mathcal{M}(v_c) & & \operatorname{Hilb}^{c,c(c-3)/2}
\end{tikzcd} \end{center}
are simultaneously well-defined.
\end{theorem}

By abuse of notation, we also denote by $\Psi$ the composition of the forgetful morphism of Theorem \ref{ModularSerreGieseker} with the isomorphism $\mathcal{M}(v_c(1))\to\mathcal{M}(v_c)$ induced by twisting by $\Oo_{\P^3}(-1)$. Thus, in this setting, $\Psi\colon\mathcal{S}^{\delta}(v_c(1))\to\mathcal{M}(v_c)$ sends $[\E(1),s]$ to $[\E]$, and its fiber over $[\E]$ is $\P H^0(\E(1))$. The morphism $\Gamma$ is applied to the twisted pair $(\E(1),s)$.

\begin{proof}
Let $\E \in \M(v_c)$. From Lemma \ref{section} it follows that $h^0(\E(1)) > 0$. Since $\E$ is $\mu$-stable, Lemma \ref{lem:quotienttorsionfree} guarantees that every pair $(\E(1),s)$ is saturated.

On the other hand, by Proposition \ref{verystable}, every $\delta$-stable pair is very stable for $0 < \delta < 1/2$, since $c_1(v_c(1))=1$ is odd and hence every semistable sheaf of class $v_c(1)$ is stable. Therefore, for any $\delta$ in the range $0 < \delta < 1/2$, the $\delta$-stable pairs are both saturated and very stable; hence, the simultaneous existence of the morphisms $\Psi$ and $\Gamma$ follows from Theorem \ref{ModularSerreGieseker} and Theorem \ref{ModularSerreHilbert}.

Finally, the calculation of the degree and genus is done using Theorem \ref{ModularSerreHilbert} noting that $v_c(1)=(1,c,c^2-2)$.
\end{proof}

\subsection{Proof of the Main Theorem, item (b)}\label{sec:proof-b}

Having explicitly computed the degree and the arithmetic genus of the subscheme $Y$ corresponding to a pair $(\E(1),s)$ in $\S^\delta(v_c(1))$ via the morphism $\Gamma$, we can now analyze the geometry of the relevant Hilbert scheme. In particular, for $d = c \geq 4$ and $g = c(c-3)/2 = \binom{c-1}{2} - 1$, obtained in the previous result, we see that $g$ is in the range: 

$$ \binom{d-1}{2}-4 < g < \binom{d-1}{2} \quad \text{and} \quad g > \binom{d-2}{2}. $$

These numerical constraints allow us to use the following irreducibility result, which also describes the general member of the component. 

\begin{theorem}\label{thm:hilbertirreducible}
The Hilbert scheme $\operatorname{Hilb}^{d,g}_{\mathbb{P} ^3}$ is irreducible whenever $(d,g)$ satisfies $d > 3$, $\binom{d-1}{2}-4 < g < \binom{d-1}{2}$, and $g > \binom{d-2}{2}$, with a general member consisting of a plane curve of degree $d$ union $\binom{d-1}{2}-g$ isolated points.
\end{theorem}

\begin{proof}
See \cite[Theorem 4.2]{Dawei2012detaching}.
\end{proof}

Therefore, as follows from the proof of \cite[Theorem 4.2]{Dawei2012detaching}, every scheme $Y\in\operatorname{Hilb}^{c,c(c-3)/2}$ fits into an exact sequence of the form
\begin{equation} \label{sqc:y}
0 \to \mathcal{I}_Y \to \mathcal{I}_C \to \Oo_p \to 0,
\end{equation}
where $C$ is a planar curve of degree $c$ and $p$ is a point.

\begin{proposition}\label{prop:gammasurj}
Let $\E \in \M(v_c)$. Then $1\leq h^0(\E(1)) \leq 2$, and $h^0(\E(1))=2$ if and only if the corresponding scheme $Y$, defined by $[Y]=\Gamma([\E(1),s])$, is planar.
\end{proposition}

\begin{proof}
The fact that $1\leq h^0(\E(1))$ follows already from Lemma \ref{section}; we prove the second inequality. Since $h^0(\E(1)) \geq 1$, there exists a non-zero section $ s \in H^0(\E(1))$, which gives the following short exact sequence 

$$0 \longrightarrow \Oo_{\mathbb{P}^3} \stackrel{s}{\longrightarrow} \E (1) \longrightarrow \mathcal{I}_Y(1) \longrightarrow 0.$$

Since $h^1(\Oo_{\P ^3})=0$, we have $$h^0(\E(1)) = h^0(\Oo_{\P ^3}) + h^0(\mathcal{I}_Y(1))=1+h^0(\mathcal{I}_Y(1)).$$
Since 
$$h^0(\mathcal{I}_Y(1)) = \left \{ \begin{matrix} 1, & \mbox{if } Y \mbox{ is planar} \\ 0, & \mbox{otherwise } \end{matrix} \right.$$
we obtain that $h^0(\E(1))\le2$, with equality holding if and only if $Y$ is planar.
\end{proof}

The next step is to show that $\Ext^1(\mathcal{I}_Y, \mathcal{O}_{\mathbb{P}^3}(-1)) \neq 0$ and to compute its dimension, which allows us to check that the morphism $\Gamma \colon \mathcal{S}^{\delta}(v_c(1))\to\operatorname{Hilb}^{c,c(c-3)/2}$ is surjective.

\begin{proposition} \label{surjectivemap}
For every scheme $Y$ satisfying an exact sequence as in \eqref{sqc:y} where $C$ is a planar curve of degree $c\ge1$, we have
$$ \operatorname{ext}^1(\mathcal{I}_Y, \Oo_{\P ^3}(-1) ) = \dfrac{1}{2}(c^2+3c).  $$
\end{proposition}

In particular, Proposition \ref{surjectivemap} applies to every $Y\in\operatorname{Hilb}^{c,c(c-3)/2}$ when $c\ge4$.

\begin{proof}
Apply the functor $\operatorname{Hom}(\cdot,\Oo_{\P^3}(-1))$ to the sequence in display \eqref{sqc:y} to see that 
$$\Ext^1(\mathcal{I}_Y, \Oo_{\P^3}(-1)) \simeq \Ext^1(\mathcal{I}_C,\Oo_{\P ^3}(-1)) $$
since $\Ext^i(\Oo_p,\Oo_{\P^3}(-1))=H^{3-i}(\Oo_p)^*=0$ for $i=1,2$.

Since $C$ is a planar curve of degree $c$, the ideal sheaf $\mathcal{I}_C$ admits the following resolution:
$$ 0 \longrightarrow \mathcal{O}_{\mathbb{P}^3}(-1-c) \longrightarrow \mathcal{O}_{\mathbb{P}^3}(-1) \oplus \mathcal{O}_{\mathbb{P}^3}(-c) \longrightarrow \mathcal{I}_C \longrightarrow 0. $$
Apply the functor $\operatorname{Hom}(\cdot,\Oo_{\P^3}(-1))$ to this exact sequence, we have
$$ 0 \to H^0(\Oo_{\P^3})\oplus H^0(\Oo_{\P^3}(c-1)) \to H^0(\Oo_{\P^3}(c)) \to \Ext^1(\mathcal{I}_C,\Oo_{\P ^3}(-1)) \to 0 $$
since
$$ \operatorname{Hom}(\mathcal{I}_C,\Oo_{\P^3}(-1))=\operatorname{Ext}^1(\Oo_{\P^3}(-1),\Oo_{\P^3}(-1))=\operatorname{Ext}^1(\Oo_{\P^3}(-c),\Oo_{\P^3}(-1))=0. $$
It follows that 
$$ \operatorname{ext}^1(\mathcal{I}_Y,\Oo_{\P ^3}(-1)) = \binom{c+3}{3}-\binom{c+2}{3}-1 = \dfrac{1}{2}(c^2+3c), $$
as desired.
\end{proof}

Since the fibers of $\Gamma$ have constant dimension, the main geometric properties of these moduli space follow. 

\begin{proposition}\label{prop:irreducible}
For $c \geq 4$ and $0 < \delta < \frac{1}{2}$, the moduli space of pairs $\mathcal{S}^{\delta}(v_c(1))$ and the Gieseker moduli space $\mathcal{M}(v_c)$ are irreducible of dimension $c^2+3c+5$.
\end{proposition}

\begin{proof}
Proposition \ref{surjectivemap} not only guarantees that $\Gamma$ is surjective (a nonzero class in $\Ext^1(\mathcal{I}_Y,\Oo_{\P^3}(-1))$ gives an extension $0\to\Oo_{\P^3}\to\E(1)\to\mathcal{I}_Y(1)\to0$ whose middle term is torsion-free, being an extension of torsion-free sheaves, and stable, since $h^0(\E)=h^0(\mathcal{I}_Y)=0$; the corresponding pair is then $\delta$-stable by Proposition~\ref{prop:stable-sheaf-stable-pair}), but also shows that its fibers 
$\mathbb{P}\Ext^1(\mathcal{I}_Y,\Oo_{\P ^3}(-1))$ have constant dimension.
Since $\operatorname{Hilb}^{c,c(c-3)/2}$ is irreducible by Theorem \ref{thm:hilbertirreducible}, we conclude that $\mathcal{S}^{\delta}(v_c(1))$ is also irreducible. In addition, since a general member of $\operatorname{Hilb}^{c,c(c-3)/2}$ is determined by a plane, a curve of degree $c$ in it and a point, so that $\dim\operatorname{Hilb}^{c,c(c-3)/2}=\binom{c+2}{2}+5$:
$$ \dim \mathcal{S}^{\delta}(v_c(1)) = \dim \operatorname{Hilb}^{c,c(c-3)/2} + \operatorname{ext}^1(\mathcal{I}_Y,\Oo_{\P ^3}(-1)) - 1 $$
$$ = 5 + \binom{c+2}{2}+\dfrac{1}{2}(c^2+3c)-1=5+c^2+3c. $$

Now we use the morphism $\Psi \colon \mathcal{S}^{\delta}(v_c(1))\to\mathcal{M}(v_c)$, which is surjective by Proposition \ref{prop:gammasurj}. Being the image of an irreducible scheme under a morphism, the Gieseker moduli space $\mathcal{M}(v_c)$ is therefore irreducible. In addition, consider the subset
$$ \M^\circ(v_c) \coloneqq \big\{[\E]\in\M(v_c) ~|~ h^0(\E(1))=1 \big\} ; $$
Proposition \ref{prop:gammasurj} implies that $\M^\circ(v_c)$ is an open subset of $\M(v_c)$ and that $\Psi^{-1}(\E)$ is a single point for every $\E\in\M^\circ(v_c)$. Therefore, $\Psi^{-1}(\M^\circ(v_c))$ is an open subset of $\S^\delta(v_c(1))$ bijective with $\M^\circ(v_c)$, hence
$$ \dim \M(v_c) = \dim \M^\circ(v_c) = \dim \Psi^{-1}(\M^\circ(v_c)) = \dim \S^\delta(v_c(1)).$$
\end{proof}

As a by-product of our arguments, we note that the \textit{Brill-Noether locus}
$$ \M'(v_c) \coloneqq \M(v_c)\setminus\M^\circ(v_c)=\big\{[\E]\in\M(v_c)~|~ h^0(\E(1))=2 \big\} $$
is an irreducible closed subscheme of codimension 2 in $\M(v_c)$.

Indeed, let $\operatorname{Hilb}_{\rm pl}^{c,c(c-3)/2}$ be the closed subset of $\operatorname{Hilb}^{c,c(c-3)/2}$ consisting of planar schemes. Proposition \ref{prop:gammasurj} implies that 
$$ \Gamma^{-1}(\operatorname{Hilb}_{\rm pl}^{c,c(c-3)/2})=\Psi^{-1}(\M'(v_c)). $$
Let $\S^\delta_{\rm pl}(v_c(1))$ denote this set. The restrictions of $\Gamma$ and $\Psi$ to $\S^\delta_{\rm pl}(v_c(1))$, denoted by $\Gamma'$ and $\Psi'$, respectively, give the following surjective maps:
\begin{center}
 \begin{tikzcd}[row sep=large, column sep=large]
    & \mathcal{S}^{\delta}_{\rm pl}(v_c(1)) \arrow[dl, "\Psi'"'] \arrow[dr, "\Gamma'"] & \\
    \mathcal{M}'(v_c) & & \operatorname{Hilb}^{c,c(c-3)/2}_{\rm pl}
\end{tikzcd} \end{center}
Since $\operatorname{Hilb}^{c,c(c-3)/2}_{\rm pl}$ is irreducible, of dimension $\binom{c+2}{2}+4$, and the fibers of $\Gamma'$ have constant dimension equal to $\operatorname{ext}^1(\mathcal{I}_Y,\Oo_{\P ^3}(-1))-1$, we get that $\mathcal{S}^{\delta}_{\rm pl}(v_c(1))$ is irreducible and 
$$ \dim \mathcal{S}^{\delta}_{\rm pl}(v_c(1)) = \dim \operatorname{Hilb}^{c,c(c-3)/2}_{\rm pl} + \operatorname{ext}^1(\mathcal{I}_Y,\Oo_{\P ^3}(-1))-1 = c^2+3c+4. $$
It then follows that $\mathcal{M}'(v_c)$ is irreducible; to compute its dimension, note that the fibers of $\Psi'$ have constant dimension equal to 1, thus
$$ \dim\mathcal{M}'(v_c) =  \dim\mathcal{S}^{\delta}_{\rm pl}(v_c(1)) - 1 = c^2+3c+3. $$

We complete this section by summarizing the global properties of $\M(v_c)$ in the following statement, which establishes part (b) of the Main Theorem.

\begin{theorem}\label{thm:item-b}
If $c \geq 4$, then $\M(v_c)$ is irreducible, generically smooth, rational, and $\dim\M(v_c)=c^2+3c+5$.
\end{theorem}

\begin{proof}
    If $c \geq 4$ then $\M(v_c)$ is irreducible by Proposition \ref{prop:irreducible}. Consequently, from Theorem \ref{0dcomp} it follows that $\M(v_c) = \operatorname{T}(-1,c,c^2,1)$, and therefore its dimension is $c^2+3c+5$. Finally, rationality and generic smoothness follow from \cite[Theorem 13(i)]{almeida2022irreducible}.
    
\end{proof}

It would be interesting to know whether $\M(v_c)$ is integral, and whether it is smooth.

\section{Recovering \texorpdfstring{$\M(v_2)$}{M2}}\label{sec:recover-2-2}

The space $\mathcal{M}(v_2)$ was studied in \cite{almeida2022irreducible}, where it is proved that it is connected, consisting of two irreducible components, namely, the closure of the moduli space of the reflexive sheaves, described by Chang \cite{chang1984small}, plus an irreducible component whose generic element is a proper torsion-free sheaf, described in Theorem \ref{0dcomp}. In this section we will use the modular Serre correspondence given in Theorems \ref{ModularSerreGieseker} and \ref{ModularSerreHilbert} to recover this result. The correspondence realizes $\mathcal{M}(v_2)$ and $\operatorname{Hilb}^{2,-1}$ as the two images of a common moduli space of pairs $\mathcal{S}^{\delta}(v_2(1))$, whose fibers are projective spaces of global sections and of extension classes, respectively. We use these morphisms and the dimensions of their fibers to establish the connectedness and the irreducible-component description of $\M(v_2)$ below.

\begin{theorem}\label{M2}
The moduli space $\mathcal{M}(v_2)$ is connected and has exactly $2$ generically smooth, rational  irreducible components, namely:
\begin{enumerate}
\item[(a)] the closure $\overline{\mathcal{R}(v_2)}$ of the family of reflexive sheaves $\R(v_2)$ of dimension $11$; 

\item[(b)] the irreducible component $\operatorname{T}(-1,2,4,1)$ given by Theorem \ref{0dcomp}, of dimension 15, whose generic element is a torsion-free sheaf $\E$ such that $\E ^{\vee \vee} \in \R (-1,2,4)$ and $\E ^{\vee \vee}/\E=\Oo_p$ for some $p\in\P^3$. 
\end{enumerate}
\end{theorem}

Theorem \ref{nowall} tells us that the morphisms $\Psi$ and $\Gamma$ are simultaneously well-defined for $0<\delta<\frac{1}{2}$. From now on, we will fix a $\delta$ in that range. By Lemma \ref{section} we can conclude that every $\E\in\M(v_2)$ admits a non-zero section $s\in H^0(\E(1))$. Therefore, we can consider the pair $(\E(1),s)$, which is $\delta$-stable by Proposition~\ref{prop:stable-sheaf-stable-pair}, hence very stable. This proves that the morphism $\Psi$ is surjective.

Theorem \ref{ModularSerreHilbert} says that the codomain of the morphism $\Gamma$ is $\operatorname{Hilb}^{2,-1}$. This Hilbert scheme was sketched in \cite{Harris1982Curves} and further elaborated in \cite{lee2000hilbert}, and \cite[Theorem 1.1]{chen2011hilbert} proved that both components are smooth; see also \cite[Section 2.1]{soulimani2024bridgeland} for a detailed description and convenient summary of all relevant references.

\begin{theorem}\label{hilb2}
The scheme $\operatorname{Hilb}^{2,-1}$ is the union of two nonsingular rational varieties $C'$ and $S$, of dimensions $11$ and $8$. In addition:
\begin{itemize}
\item[(a)] every $Y\in C'$ satisfies an exact sequence as in display \eqref{sqc:y}, where $C$ is a (possibly singular or non-reduced) conic;
\item[(b)] if $Y\in S\setminus(S\cap C')$, then $Y$ is either a union of skew lines or a double line of genus $-1$.
\end{itemize}
\end{theorem}

\begin{proof}
See \cite[Theorem 1.1]{soulimani2024bridgeland}.
\end{proof}

In the next result, we prove that the morphism given in Theorem \ref{nowall} is surjective.

\begin{lemma}\label{fiber2}
$\Gamma\colon\mathcal{S}^{\delta}(v_2(1))\rightarrow\operatorname{Hilb}^{2,-1}$ is a proper surjective morphism. 
\end{lemma}
\begin{proof}
We compute the dimension of the relevant extension group at every point of each component.

When $[Y]\in C'$, we can use Proposition \ref{surjectivemap} to conclude that $\operatorname{ext}^1(\mathcal{I}_{Y},\Oo_{\P ^3}(-1))=5$.

When $[Y]\in S\setminus(S\cap C')$, then $\Oo_Y$ fits into a short exact sequence of the form
$$ 0 \to \Oo_{L_1} \to \Oo_Y \to \Oo_{L_2} \to 0, $$
where $L_1$ and $L_2$ are (possibly equal) lines, the twist being trivial because $Y$ has genus $-1$, cf.\ \cite[Remark 2.4.1]{chang1984small}. Therefore:
$$ \Ext^1(\mathcal{I}_Y,\Oo_{\P^3}(-1)) \simeq H^2(\mathcal{I}_{Y}(-3))^* \simeq H^1(\Oo_Y(-3))^* \simeq H^1(\Oo_{L_1}(-3))^* \oplus H^1(\Oo_{L_2}(-3))^* ,$$
which gives directly that $\operatorname{ext}^1(\mathcal{I}_{Y},\Oo_{\P ^3}(-1))=4$.

Summing up:
$$ \operatorname{ext}^1({\mathcal{I}_{Y}},\Oo_{\P ^3}(-1)) = \begin{cases} 
    5, & \text{if } [Y]\in C' \\
    4, & \text{if } [Y]\in S\setminus(S\cap C'). 
\end{cases}$$
In particular, we conclude that 
$\Gamma \colon  \mathcal{S}^{\delta}(v_2(1))  \rightarrow  \operatorname{Hilb}^{2,-1}$ is surjective.
\end{proof}

We now use this lemma to relate the geometry of the Hilbert scheme to that of the moduli space of pairs.

\begin{proposition}\label{S22}
For $0 < \delta < \frac{1}{2}$ the moduli space $\mathcal{S}^{\delta}(v_2(1))$ is connected and has two irreducible components, of dimensions $11$ and $15$.
\end{proposition}

\begin{proof}
Set $\mathcal{S}_1\coloneqq\Gamma^{-1}(C')$ and $\mathcal{S}_2\coloneqq\overline{\Gamma^{-1}(S\setminus(S\cap C'))}$; we argue that both closed subsets of $\mathcal{S}^{\delta}(v_2(1))$ are irreducible.

First, we noted in the proof of Lemma \ref{fiber2} that the fibers of the restricted morphism $\mathcal{S}_1\to C'$ have constant dimension equal to $4$. Since $\mathcal{S}_1\to C'$ is surjective, and $C'$ is irreducible, we conclude that $\mathcal{S}_1$ is irreducible of dimension $\dim C'+4=15$.

Similarly, $\Gamma^{-1}([Y])\simeq\P^3$ for every $Y\in S\setminus(S\cap C')$, thus $\overline{\Gamma^{-1}(S\setminus(S\cap C'))}$ is irreducible of dimension equal to $\dim S+3=11$.

Since $\Gamma$ is surjective, we have 
$$ \mathcal{S}^{\delta}(v_2(1))=\Gamma^{-1}(C')\cup\Gamma^{-1}(S\setminus(S\cap C'))\subseteq\mathcal{S}_1\cup\mathcal{S}_2 \subseteq \mathcal{S}^{\delta}(v_2(1)),$$
thus $\mathcal{S}^{\delta}(v_2(1))=\mathcal{S}_1\cup\mathcal{S}_2$. Neither is contained in the other: $\dim\mathcal{S}_1>\dim\mathcal{S}_2$, while $\Gamma(\mathcal{S}_1)\subseteq C'$ and $\Gamma(\mathcal{S}_2)$ contains points of $S\setminus(S\cap C')$.

As for connectedness, note that $\operatorname{Hilb}^{2,-1}$ is connected. The morphism $\Gamma \colon \mathcal{S}^{\delta}(v_2(1)) \to \operatorname{Hilb}^{2,-1}$ is a proper surjective map whose fibers are projective spaces, which are connected. Since the base and the fibers of this morphism are connected, it follows that $\mathcal{S}^{\delta}(v_2(1))$ is connected.
\end{proof}

\subsection{Proof of Theorem \ref{M2} and Main Theorem item (c).} We can now check the connectedness of $\M(v_2)$: $ \mathcal{S}^{\delta}(v_2(1))$ is connected by Proposition \ref{S22} and the morphism $\Psi\colon \mathcal{S}^{\delta}(v_2(1)) \to \M(v_2)$ is surjective, thus $\M(v_2)$ is connected.
The source $\mathcal{S}^{\delta}(v_2(1))$ is proper over $\mathbb{C}$ by Lemma \ref{CorGammaProper} and the projectivity of $\operatorname{Hilb}^{2,-1}$. Since $\M(v_2)$ is separated over $\mathbb{C}$, the morphism $\Psi$ is proper. It is also surjective, as established above, so the image of an irreducible component in $\mathcal{S}^{\delta}(v_2(1))$ is an irreducible closed subset of $\M(v_2)$; hence, to prove that $\M(v_2)$ has exactly two irreducible components, it suffices to argue that the image of one component is not contained in the image of the other. 

Let $\mathcal{S}^{\delta}(v_2(1))=\mathcal{S}_1 \cup \mathcal{S}_2$ be the decomposition of $\mathcal{S}^{\delta}(v_2(1))$ into its two irreducible components, as presented in the proof of Proposition \ref{S22}. 

Assume first that  $\Psi(\mathcal{S}_1)\subset\Psi(\mathcal{S}_2) = \M(v_2)$, and take $(\E(1),s)\in\mathcal{S}_1\setminus\mathcal{S}_2$; we can choose $\E$ such that $h^0(\E(1))=1$ since a generic point $[Y]\in C'$ represents a nonplanar scheme. Then there is $(\E'(1),s')\in\mathcal{S}_2$ such that $\Psi(\E(1),s)=\Psi(\E'(1),s')$, thus $\E\simeq\E'$ and one can assume that $s,s'\in H^0(\E(1))$. It follows that $s=\lambda s'$ for some $\lambda\in\mathbb{C}^*$, thus $(\E(1),s)\simeq(\E'(1),s')$, contradicting the hypothesis that $(\E(1),s)\in\mathcal{S}_1\setminus\mathcal{S}_2$ while $(\E'(1),s')\in\mathcal{S}_2$.

Next, assume that ${\Psi}(\mathcal{S}_2)\subset\Psi(\mathcal{S}_1) = \M(v_2)$, and take $(\E(1),s)\in\mathcal{S}_2\setminus\mathcal{S}_1$; note that no curve $Y\in S\setminus(S\cap C')$ is planar, so $h^0(\E(1))=1$. We then obtain a contradiction similar to the previous case.

Finally, we identify the two components. Since $\Psi$ is injective on the dense open subsets of $\mathcal{S}_1$ and $\mathcal{S}_2$ where $h^0(\E(1))=1$, we have $\dim\Psi(\mathcal{S}_1)=15$ and $\dim\Psi(\mathcal{S}_2)=11$. On the other hand, $\R(v_2)$ is irreducible of dimension $11$ by \cite{chang1984small}, and its closure is not contained in $\operatorname{T}(-1,2,4,1)$, whose generic point is not reflexive; hence $\overline{\R(v_2)}$ is an irreducible component of dimension $11$, and $\operatorname{T}(-1,2,4,1)$ is one of dimension $15$ by Theorem~\ref{0dcomp}. Therefore $\Psi(\mathcal{S}_2)=\overline{\R(v_2)}$ and $\Psi(\mathcal{S}_1)=\operatorname{T}(-1,2,4,1)$.

\subsection{Spectra of sheaves}\label{sec:spectra-v2}

Given $[\E]\in\M(v_2)$, set $s_\E\coloneqq h^0(\lext^2(\E,\Oo_{\Pt}))$. We
show that the pairs $(s_\E,\text{spectrum of }\E)$ that occur are exactly
\[
\bigl(0,(-1,-1)\bigr)\qquad\text{and}\qquad\bigl(1,(-2,-1)\bigr),
\]
and that for every such sheaf
\[
h^1(\E(t))=s_\E\quad(t\leq-1),\qquad h^2(\E(t))=0\quad(t\geq0),
\qquad h^0(\E(1))=1+h^1(\E(1))\geq1.
\]

Indeed, by \eqref{eq:spectrum-sum-minus-one} with $c_2=c_3=2$ we have
$k_1+k_2=-2-s_\E$, and by Proposition~\ref{charspec}(b) both entries are
at least $-2$, so that $s_\E\leq2$. If $s_\E=2$, then $k_1=k_2=-2$, and
Proposition~\ref{charspec}(b) would force $-1$ to occur as well, which is
impossible. If $s_\E=1$, then $k_1+k_2=-3$ and the only possibility is
$(-2,-1)$. If $s_\E=0$, then $k_1+k_2=-2$, leaving $(-2,0)$ and $(-1,-1)$;
the first is again excluded by Proposition~\ref{charspec}(b). Both
remaining pairs occur: every $[\E]\in\R(v_2)$ has $s_\E=0$, hence spectrum
$(-1,-1)$, while the kernel of an epimorphism $\F\to\Oo_p$ with
$[\F]\in\R(-1,2,4)$ and $p\notin\sing(\F)$, as in Theorem~\ref{0dcomp},
has $s_\E=1$, as follows from the local Ext sequence of $0\to\E\to\F\to\Oo_p\to0$.

The cohomology formulas follow by substituting the two spectra into
Theorem~\ref{thm-spectrum}. Finally, \eqref{eq:RR-spectrum} gives
$\chi(\E(1))=1$, and $h^3(\E(1))=0$ by stability, so the vanishing of
$h^2(\E(1))$ proves the last assertion.

\subsection{Classifying non-reflexive sheaves in \texorpdfstring{$\M(v_2)$}{M2}}

Let us take this opportunity to rectify the claim in \cite[Theorem 26(iii)]{almeida2022irreducible}. The proof of \cite[Theorem 26]{almeida2022irreducible} correctly classifies all non-reflexive sheaves $\E\in\M(v_2)$, which fall in three distinct families:

\begin{itemize}
\item[(I)] $\E=\ker\{\F\onto\Oo_p\}$, with $\F\in\R(-1,2,4)$ and $p$ is a point; every such sheaf belongs to the component $\operatorname{T}(-1,2,4,1)$.
\item[(II)] $\E=\ker\{\F\onto\Oo_L\}$, with $\F\in\R(-1,1,1)$ and $L$ is a line; this family is denoted by $\operatorname{X}(-1,1,1,0,0)$ in the notation of \cite{almeida2022irreducible};
\item[(III)] $\E=\ker\{\F\onto\mathcal{Q}\}$, with $\F\in\R(-1,1,1)$ and $\mathcal{Q}$ fitting into an exact sequence of the form 
\[ 0 \to \Oo_p \to \mathcal{Q} \to \Oo_L(-1) \to 0 ; \]
this family is denoted by $\operatorname{X}(-1,1,1,-1,1)$ in the notation of \cite{almeida2022irreducible};
\end{itemize}

The latter family is an irreducible scheme of dimension 11, Theorem \ref{Xscheme}, fully contained in the component $\operatorname{T}(-1,2,4,1)$, see \cite[Proposition 25(i)]{almeida2022irreducible}.

The family $\operatorname{X}(-1,1,1,0,0)$ is an irreducible scheme of dimension 9, Theorem \ref{Xscheme}; we claim that it is contained in $\overline{\R(v_2)}$, and that its generic point does not lie in $\operatorname{T}(-1,2,4,1)$. Indeed, the generic point $[\E] \in \operatorname{X}(-1,1,1,0,0)$ satisfies $\dim \Ext^1(\E,\E) = 11$ by Theorem \ref{Xscheme}; therefore, $[\E]\notin \operatorname{T}(-1,2,4,1)$, since every closed point of the $15$-dimensional component $\operatorname{T}(-1,2,4,1)$ has Zariski tangent space of dimension at least $15$. Hence the generic point of $\operatorname{X}(-1,1,1,0,0)$ lies in $\overline{\R(v_2)}$; since $\operatorname{X}(-1,1,1,0,0)$ is irreducible and $\overline{\R (v_2)}$ is closed, $\operatorname{X}(-1,1,1,0,0) \subset \overline{\R (v_2)}$. Note also that $\R(v_2)\cap \operatorname{T}(-1,2,4,1)=\varnothing$: the argument in the proof of Proposition~\ref{prop:boundary-v3} showing that $\operatorname{T}(-1,3,9,1)$ consists of sheaves of homological dimension $2$ applies verbatim to $\operatorname{T}(-1,2,4,1)$, while reflexive sheaves have homological dimension at most $1$.

We conclude:

\begin{proposition}\label{prop:boundary-v2}
$$ \overline{\R(v_2)}\setminus\R(v_2)=\operatorname{X}(-1,1,1,0,0) ~\cup~ \overline{\R(v_2)}\cap \operatorname{T}(-1,2,4,1). $$
\end{proposition}
\begin{proof} Every non-reflexive sheaf in $\M(v_2)$ belongs to one of the families (I), (II), (III); the first and the third lie in $\operatorname{T}(-1,2,4,1)$, and the second is $\operatorname{X}(-1,1,1,0,0)\subset\overline{\R(v_2)}$. This gives the inclusion $\subseteq$. The reverse inclusion follows from $\operatorname{X}(-1,1,1,0,0)\subset\overline{\R(v_2)}$ and $\R(v_2)\cap \operatorname{T}(-1,2,4,1)=\varnothing$. \end{proof}

It would be interesting to determine whether the intersection $\overline{\R(v_2)}\cap \operatorname{T}(-1,2,4,1)$ is irreducible and compute its dimension.

\section{Description of \texorpdfstring{$\M(v_3)$}{M3}}\label{sec:M(-1,3,7)}

The moduli space $\M(v_3)$ is known to contain at least two irreducible components: the closure of the moduli space of stable rank 2 reflexive sheaves constructed by Chang \cite{chang1984small}, and the T-component constructed in Theorem \ref{0dcomp}. In this section, we will prove that, in fact, these components are the only two and complete the proof of item (d) of the Main Theorem \ref{maintheorem}. Once again, we will use the modular Serre correspondence and a technique similar to that of the previous section.

\begin{theorem}\label{M3}
The moduli space $\M(v_3)$ is connected and has exactly two irreducible components, namely:

\begin{enumerate}
\item[(a)] the closure $\overline{{\R(v_3)}}$ of the family of reflexive sheaves ${\R(v_3)}$, which is rational of dimension $19$ and generically smooth;

\item[(b)] the irreducible component $\operatorname{T}(-1,3,9,1)$ given by Theorem \ref{0dcomp}, of dimension $23$, whose generic element is a torsion-free sheaf $\E$ such that $\E ^{\vee \vee} \in {\mathcal{R}(-1,3,9)}$ and $\E ^{\vee \vee}/\E=\Oo_p$ for some $p\in\P^3$.

\end{enumerate}
\end{theorem}

Item (a) is an immediate consequence of the following claim.

\begin{theorem}\label{thm:R3smooth}
${\R(v_3)}$ is irreducible, smooth, and rational of dimension $19$.
\end{theorem}
\begin{proof}
Chang proves in \cite[Theorem 3.15]{chang1984small} that $\R(v_3)$ is irreducible, generically \textit{reduced}, and rational of dimension $19$; we add that $\R(v_3)$ is \textit{smooth}.

Indeed, it is enough to check that $\Ext^2(\F,\F)=0$ for every $\F$ in $\R(v_3)$. Note from \cite[Table 3.15.1]{chang1984small} that any $\F\in \R(v_3)$ is 2-regular and $h^0({\F}(-1))=h^1({\F}(-2))=0$; we can then apply \cite[Lemma 18]{JMun25} to conclude that $\Ext^2({\F},\F)=0$, as desired.
\end{proof}

By Theorem \ref{nowall}, the morphisms $\Psi$ and $\Gamma$ are simultaneously well-defined for $0<\delta<\frac{1}{2}$. From now on, we will fix a $\delta$ in that range. By Lemma \ref{section}, for every $\E$ with Chern classes $(-1,3,7)$ there exists a non-zero section $s \in H^0(\E(1))$. Therefore, we can consider the pair $(\E(1),s)$, which is $\delta$-stable by Proposition~\ref{prop:stable-sheaf-stable-pair}, hence very stable. This proves that the morphism $\Psi$ is surjective.

To study the morphism $\Gamma$, we first need to understand the associated Hilbert scheme. Using Theorem \ref{ModularSerreHilbert}, we find that the codomain of $\Gamma$ is $\operatorname{Hilb}^{3,0}$. Piene and Schlessinger described this Hilbert scheme. 

\begin{theorem}\label{hilb3}
The scheme $\operatorname{Hilb}^{3,0}$ is the union of two nonsingular rational varieties $H$ and $H'$, of dimensions $12$ and $15$, respectively; 
\begin{itemize}
\item[(a)] $H$ is the closure of the set of twisted cubics
\item[(b)] $H'$ is the closure of the set of unions of a smooth plane cubic curve and a point not on it.
\end{itemize}
In addition, their transversal intersection is nonsingular, rational, and has dimension $11$.
\end{theorem}

\begin{proof}
See \cite[Section 6 Theorem]{piene1985hilbert}.
\end{proof}

We are now ready to check that $\Gamma$ is surjective.

\begin{lemma}\label{fiber3}
$\Gamma\colon\mathcal{S}^{\delta}(v_3(1))\rightarrow\operatorname{Hilb}^{3,0}$ is a proper surjective morphism. 
\end{lemma}

\begin{proof}
When $[Y]\in H'$, we can use Proposition \ref{surjectivemap} to conclude that
$$ \operatorname{ext}^1(\mathcal{I}_{Y},\Oo_{\P ^3}(-1))=9. $$

When $[Y]\in H\setminus(H\cap H')$, it follows that $\mathcal{I}_{Y}$ admits the resolution
$$ 0 \to \opThree(-3)^{\oplus2} \to \opThree(-2)^{\oplus3} \to \mathcal{I}_{Y} \to 0.$$
Applying the functor $\Hom(\cdot,\opThree(-1))$ and taking the associated long exact sequence of cohomology, we get
$$ 0 \to H^0(\opThree(1))^{\oplus3} \to H^0(\opThree(2))^{\oplus2} \to \Ext^1(\mathcal{I}_{Y},\opThree(-1)) \to 0$$
thus 
$$\operatorname{ext} ^1(\mathcal{I}_{Y}, \Oo_{\P^3}(-1)) = 2h^0( \mathcal{O}_{\mathbb{P}^3}(2))  - 3h^0( \Oo _{\P^3}(1) ) = 20 - 12 =8 .$$

Summing up:
$$ \operatorname{ext}^1(\mathcal{I}_{Y},\Oo_{\P ^3}(-1)) = \begin{cases} 
    9, & \text{if } [Y]\in H' \\
    8, & \text{if } [Y]\in H\setminus(H\cap H'). 
\end{cases}$$
In particular, we conclude that 
$\Gamma \colon  \mathcal{S}^{\delta}(v_3(1))  \rightarrow  \operatorname{Hilb}^{3,0}$ is surjective.
\end{proof}

\begin{proposition}\label{S32}
For $0 < \delta < \frac{1}{2}$ the moduli space $\mathcal{S}^{\delta}(v_3(1))$ is connected and has two irreducible components of dimensions $19$ and $23$. 
\end{proposition}

\begin{proof}
We follow the same argument as in the proof of Proposition \ref{S22}. Set $\mathcal{S}_1\coloneqq\Gamma^{-1}(H')$ and $\mathcal{S}_2\coloneqq\overline{\Gamma^{-1}(H\setminus(H\cap H'))}$; we argue that both closed subsets of $\mathcal{S}^{\delta}(v_3(1))$ are irreducible.

First, Lemma \ref{fiber3} shows that $\Gamma^{-1}([Y])\simeq\mathbb{P}^8$ for every $Y\in H'$. Therefore, $\mathcal{S}_1\coloneqq\Gamma^{-1}(H')$ is irreducible of dimension $\dim H'+8=23$.

Similarly, $\Gamma^{-1}({[Y]})\simeq\mathbb{P}^7$ for every $Y\in H\setminus(H\cap H')$, thus $\Gamma^{-1}(H\setminus(H\cap H'))$ is irreducible of dimension $12+7=19$. Therefore, $\dim\mathcal{S}_2=19$.

Since $\Gamma$ is surjective, we have 
$$ \mathcal{S}^{\delta}(v_3(1))=\Gamma^{-1}(H')\cup\Gamma^{-1}(H\setminus(H\cap H'))\subseteq\mathcal{S}_1\cup\mathcal{S}_2 \subseteq \mathcal{S}^{\delta}(v_3(1)),$$
thus $\mathcal{S}^{\delta}(v_3(1))=\mathcal{S}_1\cup\mathcal{S}_2$. Neither is contained in the other: $\dim\mathcal{S}_1>\dim\mathcal{S}_2$, while $\Gamma(\mathcal{S}_1)\subseteq H'$ and $\Gamma(\mathcal{S}_2)$ contains points of $H\setminus(H\cap H')$.

As for connectedness, note that $\operatorname{Hilb}^{3,0}$ is connected. The morphism $\Gamma \colon \mathcal{S}^{\delta}(v_3(1)) \to \operatorname{Hilb}^{3,0}$ is a proper surjective map whose fibers are projective spaces, which are connected. Since the base and the fibers of this morphism are connected, it follows that $\mathcal{S}^{\delta}(v_3(1))$ is connected.
\end{proof}

To complete the proof of Theorem \ref{M3}, we need to prove that $\M(v_3)$ has exactly two components just like $\mathcal{S}^{\delta}(v_3(1))$, and understand what the Hilbert scheme associated with the sheaves in each component is.

\subsection{Proof of Theorem \ref{M3} and Main Theorem item (d).}

We can now check the connectedness of $\M(v_3)$: $\mathcal{S}^{\delta}(v_3(1))$ is connected by Proposition \ref{S32} and the morphism $\Psi\colon \mathcal{S}^{\delta}(v_3(1)) \to \M(v_3)$ is surjective by Lemma \ref{section}, thus $\M(v_3)$ is connected.

The source $\mathcal{S}^{\delta}(v_3(1))$ is proper over $\mathbb{C}$ by Lemma \ref{CorGammaProper} and the projectivity of $\operatorname{Hilb}^{3,0}$. Since $\M(v_3)$ is separated over $\mathbb{C}$, the morphism $\Psi$ is proper, as well as surjective. Thus the images of the two irreducible components of $\mathcal{S}^{\delta}(v_3(1))$ are irreducible closed subsets whose union is $\M(v_3)$. To prove that these are exactly the two irreducible components of $\M(v_3)$, it suffices to show that neither image is contained in the other.

Let $\mathcal{S}^{\delta}(v_3(1))=\mathcal{S}_1 \cup \mathcal{S}_2$ be the decomposition of $\mathcal{S}^{\delta}(v_3(1))$ into its two irreducible components, as presented in the proof of Proposition \ref{S32}. 

Assume first that $\Psi(\mathcal{S}_1)\subset\Psi(\mathcal{S}_2) = \M(v_3)$, and take $(\E(1),s)\in\mathcal{S}_1\setminus\mathcal{S}_2$; we can choose $\E$ such that $h^0(\E(1))=1$ since a generic point ${[Y]}\in H'$ {corresponds to a nonplanar scheme}, cf. Proposition \ref{prop:gammasurj}. Then there is $(\E'(1),s')\in\mathcal{S}_2$ such that $\Psi(\E(1),s)=\Psi(\E'(1),s')$, thus $\E\simeq\E'$ and one can assume that $s,s'\in H^0(\E(1))$. Since $h^0(\E(1))=1$, it follows that $s=\lambda s'$ for some $\lambda\in\mathbb{C}^*$, thus $(\E(1),s)\simeq(\E'(1),s')$, contradicting the hypothesis that $(\E(1),s)\in\mathcal{S}_1\setminus\mathcal{S}_2$ while $(\E'(1),s')\in\mathcal{S}_2$.

Next, assume that $\Psi(\mathcal{S}_2)\subset \Psi(\mathcal{S}_1) = \M(v_3)$, and take $(\E(1),s)\in\mathcal{S}_2\setminus\mathcal{S}_1$; note that no curve $Y\in H\setminus(H\cap H')$ is planar, so $h^0(\E(1))=1$. We then obtain a contradiction similar to the previous case.

Finally, we identify the two components. Since $\Psi$ is injective on the dense open subsets of $\mathcal{S}_1$ and $\mathcal{S}_2$ where $h^0(\E(1))=1$, we have $\dim\Psi(\mathcal{S}_1)=23$ and $\dim\Psi(\mathcal{S}_2)=19$. On the other hand, $\overline{\R(v_3)}$ is an irreducible component of dimension $19$ by Theorem~\ref{thm:R3smooth}, since its generic point is reflexive while that of $\operatorname{T}(-1,3,9,1)$ is not, and $\operatorname{T}(-1,3,9,1)$ is an irreducible component of dimension $23$ by Theorem~\ref{0dcomp}. Therefore $\Psi(\mathcal{S}_2)=\overline{\R(v_3)}$ and $\Psi(\mathcal{S}_1)=\operatorname{T}(-1,3,9,1)$.

\subsection{Spectra of sheaves}\label{sec:spectra-v3}
For $[\E]\in \M(v_3)$, we will now show that the possible pairs consisting
of $s_\E=h^0(\lext^2(\E,\Oo_{\Pt}))$ and the spectrum are exactly
\[
\bigl(0,(-2,-2,-1)\bigr),\qquad\bigl(1,(-3,-2,-1)\bigr).
\]
For every such sheaf,
\[
h^1(\E(t))=s_\E\quad(t\leq-1),\qquad
h^2(\E(t))=0\quad(t\geq1),
\]
and the remaining values in the spectrum range are
\[
h^2(\E(-2))=5+s_\E,\qquad
h^2(\E(-1))=2+s_\E,\qquad h^2(\E)=s_\E.
\]
Moreover, $h^0(\E(1))=1+h^1(\E(1))\geq1$.

Indeed, by \eqref{eq:spectrum-sum-minus-one}, the three spectral integers have
sum $-5-s_\E$. Each is at least $-3$ by
Proposition \ref{charspec}. If $-3$ occurs,
connectedness uses all three positions, giving $(-3,-2,-1)$ and
$s_\E=1$. Otherwise all entries are at least $-2$, so their sum is at least $-6$ and $s_\E\leq1$. If $s_\E=1$, the spectrum would be $(-2,-2,-2)$, which violates Proposition~\ref{charspec}(b); hence $s_\E=0$, and the only spectrum with entries at least $-2$, sum $-5$, and satisfying Proposition~\ref{charspec}(b) is $(-2,-2,-1)$. Both pairs occur: every $[\E]\in\R(v_3)$ has $s_\E=0$, hence spectrum $(-2,-2,-1)$. For $s_\E=1$,
take an elementary transformation of a sheaf in $\R(-1,3,9)$ at a
point where it is locally free, as in Theorem~\ref{0dcomp}.

The cohomology formulas follow by substituting the two spectra into
Theorem~\ref{thm-spectrum}. Since $\chi(\E(1))=1$ by
\eqref{eq:RR-spectrum}, and stability gives $h^3(\E(1))=0$, the
vanishing of $h^2(\E(1))$ proves the last assertion.

\subsection{Classifying non-reflexive sheaves in \texorpdfstring{$\M(v_3)$}{M3}}
\label{sec:boundary-v3}

For a non-reflexive sheaf $[\E]\in\M(v_3)$, write
\[
0\longrightarrow\E\longrightarrow\F\longrightarrow\mathcal Q
\longrightarrow0,\qquad\F=\E^{\vee\vee}.
\]
The hull $\F$ is slope-stable and the nonzero quotient $\mathcal Q$
has support of dimension at most $1$. If its one-dimensional part has
multiplicity $\operatorname{mult}(\mathcal Q)$, then $c_2(\F)=3-\operatorname{mult}(\mathcal Q)\geq1$ by \cite[Proposition 14]{almeida2022irreducible} and \cite[Corollary 3.3]{hartshorne1980reflexive}, so $\operatorname{mult}(\mathcal Q)\leq2$.
Let $\mathcal Z$ denote the maximal zero-dimensional subsheaf of
$\mathcal Q$. Since $\F$ is reflexive and $\mathcal Q/\mathcal Z$ is
either zero or pure one-dimensional, hence Cohen--Macaulay, the local Ext sequences give
\[
\lext^2(\E,\Oo_{\Pt})\simeq\lext^3(\mathcal Q,\Oo_{\Pt})
\simeq\lext^3(\mathcal Z,\Oo_{\Pt}).
\]
Consequently,
\begin{equation}\label{eq:length}
s_\E=\operatorname{length}(\mathcal Z)\in\{0,1\},
\end{equation}
where the last assertion follows from Section~\ref{sec:spectra-v3}.
Thus a nonzero pure one-dimensional quotient corresponds to
homological dimension $1$, whereas $\mathcal Z\neq0$ corresponds to
homological dimension $2$. We first analyze the pure multiplicity-two
case, and then prove the complete boundary formula.

\subsubsection*{Cohomology of the multiplicity-two quotient}

\begin{proposition}\label{prop:quotient-cohomology}
Let $[\E]\in\M(v_3)$ satisfy $s_\E=0$,
and suppose that $\mathcal Q=\E^{\vee\vee}/\E$ has multiplicity $2$.
Then $\mathcal Q$ is pure one-dimensional, its reflexive hull belongs to
$\R(-1,1,1)$, and
\[ P_{\mathcal Q}(t)=\chi(\mathcal Q(t))=2t,
\qquad H^0(\mathcal Q)=H^1(\mathcal Q)=0. \]
\end{proposition}

\begin{proof}
Set $\F=\E^{\vee\vee}$, so that
\[
0\longrightarrow\E\longrightarrow\F\longrightarrow\mathcal Q
\longrightarrow0.
\]
By \eqref{eq:length}, the hypothesis $s_\E=0$ means that $\mathcal Q$
has no nonzero zero-dimensional subsheaf. Since its support has dimension
at most $1$ and its multiplicity is $2$, it is pure one-dimensional.

The reflexive hull $\F$ is slope-stable, with $c_1(\F)=-1$ and
\[
c_2(\F)=c_2(\E)-\operatorname{mult}(\mathcal Q)=3-2=1.
\]
The third Chern class bound for stable rank $2$ reflexive sheaves and the parity condition, see \cite[Corollary 2.4 and Theorem 8.2]{hartshorne1980reflexive}, give
\[
0\leq c_3(\F)\leq c_2(\F)^2=1,
\qquad c_3(\F)\equiv c_1(\F)c_2(\F)\equiv1\pmod2;
\]
thus $c_3(\F)=1.$ Subtracting the Riemann--Roch polynomials in \eqref{eq:RR-spectrum}, we obtain
\[
P_{\mathcal Q}(t) = P_{\F}(t)-P_{\E}(t) = 2t \\
\]
In particular, $\chi(\mathcal Q)=0$.

Section~\ref{sec:spectra-v3}, with $s_\E=0$,
gives $h^2(\E)=0$. Stability and Serre duality give
$h^0(\E)=h^3(\E)=0$, while Riemann--Roch gives $\chi(\E)=0$.
Consequently, $h^1(\E)=0$. Since $h^0(\F)=0$ by stability, the exact
sequence
\[
H^0(\F)\longrightarrow H^0(\mathcal Q)\longrightarrow H^1(\E)
\]
forces $H^0(\mathcal Q)=0$. A sheaf with one-dimensional support has no
cohomology in degrees at least $2$, so
\[
0=\chi(\mathcal Q)=h^0(\mathcal Q)-h^1(\mathcal Q)
\]
also gives $H^1(\mathcal Q)=0$.

\end{proof}

\begin{definition}\label{def:mult-two-locus}
We write $\mathcal C(-1,1,1,2)$ for the locus of those $[\E]\in\M(v_3)$ whose
quotient $\mathcal Q=\E^{\vee\vee}/\E$ is pure of dimension one and multiplicity
two.  
\end{definition}

By Proposition~\ref{prop:quotient-cohomology}, every sheaf $[\E]\in\mathcal C(-1,1,1,2)$ satisfies
$[\E^{\vee\vee}]\in\R(-1,1,1)$ and $P_{\mathcal Q}(t)=2t$.

\subsubsection*{Self-extensions of the multiplicity-two kernels}
\label{sec:self-extensions}

We now exhibit a locally free resolution of these kernels and deduce that they are smooth points of $\M(v_3).$

\begin{proposition}\label{prop:multiplicity-two-self-ext}
Every $[\E]\in\mathcal C(-1,1,1,2)$ admits a locally free
resolution
\begin{equation}\label{eq:kernel-minimal-resolution}
0\longrightarrow\Oo_{\Pt}(-3)^{\oplus2}
\longrightarrow\Oo_{\Pt}(-2)^{\oplus3}\oplus\Oo_{\Pt}(-1)
\longrightarrow\E\longrightarrow0.
\end{equation}
Consequently,
\[
\operatorname{Ext}^2(\E,\E)=\operatorname{Ext}^3(\E,\E)=0,
\qquad \dim\operatorname{Ext}^1(\E,\E)=19.
\]
Every such sheaf is an unobstructed, smooth point of $\M(v_3)$.
Moreover,
\[
\mathcal C(-1,1,1,2)
\subseteq\overline{\R(v_3)}
\setminus\bigl(\R(v_3)\cup \operatorname{T}(-1,3,9,1)\bigr).
\]
\end{proposition}

\begin{proof}
Write the defining sequence as
\[
0\longrightarrow\E\longrightarrow\F\xrightarrow{\varphi}\mathcal Q
\longrightarrow0,
\qquad [\F]\in\R(-1,1,1).
\]
Proposition~\ref{prop:quotient-cohomology} gives
$H^0(\mathcal Q)=H^1(\mathcal Q)=0$ and $P_{\mathcal Q}(t)=2t$.
Thus $\mathcal Q(2)$ is a rank $0$ instanton sheaf of degree $2$.
Applying \cite[Proposition 3]{jardim2017moduli} and twisting by $-2$
gives a resolution
\begin{equation}\label{eq:quotient-linear-resolution}
0\longrightarrow\Oo_{\Pt}(-3)^{\oplus2}
\xrightarrow{d_2}\Oo_{\Pt}(-2)^{\oplus4}
\xrightarrow{d_1}\Oo_{\Pt}(-1)^{\oplus2}
\xrightarrow{\epsilon}\mathcal Q\longrightarrow0.
\end{equation}
In particular, $\epsilon$ induces an isomorphism on global sections
after twisting by $1$.

The hull $\F\in\R(-1,1,1)$ has the resolution
\begin{equation}\label{eq:hull-resolution}
0\longrightarrow\Oo_{\Pt}(-2)
\xrightarrow{\alpha=(\ell_1,\ell_2,\ell_3)}
\Oo_{\Pt}(-1)^{\oplus3}\longrightarrow\F\longrightarrow0,
\end{equation}
where $\ell_1,\ell_2,\ell_3$ are independent linear forms; see the proof of \cite[Lemma 9.3]{hartshorne1980reflexive}, and note that dependent forms would have a common line of zeros, making the cokernel non-reflexive. The composite map from
$\Oo_{\Pt}(-1)^{\oplus3}$ to $\mathcal Q$, obtained via $\F$,
lifts uniquely to a constant matrix
\[
\beta \colon \Oo_{\Pt}(-1)^{\oplus3}\longrightarrow
\Oo_{\Pt}(-1)^{\oplus2}.
\]
Indeed, setting $\mathcal K_Q\coloneqq\ker\epsilon$, the resolution \eqref{eq:quotient-linear-resolution} twisted by $1$ gives $H^0(\mathcal K_Q(1))=H^1(\mathcal K_Q(1))=0$, which is exactly what is needed for the existence and uniqueness of the lift. Since $\varphi$ is surjective, $\beta\neq0$, so $1\leq\operatorname{rank}\beta\leq2$. If the rank of $\beta$ were $1$, the composite
would factor through an epimorphism $\Oo_{\Pt}(-1)\to\mathcal Q$.
Consequently, $\mathcal Q\simeq\Oo_Z(-1)$ for a closed subscheme $Z$.
The relation $\epsilon\beta\alpha=0$ would imply that $Z$ is annihilated
by a nonzero linear combination of the independent forms
$\ell_1,\ell_2,\ell_3$. Hence $Z$ lies in a plane. Purity and
multiplicity $2$ then make $Z$ a plane conic: a pure one-dimensional
subscheme of a smooth plane is an effective Cartier divisor, here of
degree $2$. But $P_{\Oo_Z(-1)}(t)=2t-1$, contradicting
$P_{\mathcal Q}(t)=2t$. Then  $\operatorname{rank}\beta = 2.$

Let $\mathcal A$ be the kernel of
$\Oo_{\Pt}(-1)^{\oplus3}\to\mathcal Q$. Since the constant matrix
$\beta$ is surjective, it splits and
\[
\mathcal A\simeq\Oo_{\Pt}(-1)\oplus\mathcal K_Q.
\]
Resolution \eqref{eq:quotient-linear-resolution} therefore gives
\[
0\longrightarrow\Oo_{\Pt}(-3)^{\oplus2}
\longrightarrow\Oo_{\Pt}(-2)^{\oplus4}\oplus\Oo_{\Pt}(-1)
\longrightarrow\mathcal A\longrightarrow0.
\]
The inclusion $\alpha\colon\Oo_{\Pt}(-2)\to\mathcal A$ has cokernel $\E$.
It lifts to the middle term of this last resolution, because
$H^1(\Oo_{\Pt}(-1))=0$. Thus
\[
0\longrightarrow\Oo_{\Pt}(-3)^{\oplus2}\oplus\Oo_{\Pt}(-2)
\longrightarrow\Oo_{\Pt}(-2)^{\oplus4}\oplus\Oo_{\Pt}(-1)
\longrightarrow\E\longrightarrow0
\]
is exact. The component of this lift into
$\Oo_{\Pt}(-2)^{\oplus4}$ is a constant vector $w$ satisfying
\[
d_1w=\beta\alpha.
\]
It is nonzero: if $w=0$, then $\beta\alpha=0$, so $(\ell_1,\ell_2,\ell_3)$ would be proportional to a constant vector spanning $\ker\beta$, contradicting the independence of the $\ell_i$. We can therefore cancel one isomorphic pair of $\Oo_{\Pt}(-2)$
summands. The resulting resolution is
\eqref{eq:kernel-minimal-resolution}, which proves our first claim. 

Now the Ext group computations are standard diagram chasing after applying
$\operatorname{Hom}(-,\E)$ to the resolution
\eqref{eq:kernel-minimal-resolution}: one gets
$\operatorname{Ext}^2(\E,\E)=\operatorname{Ext}^3(\E,\E)=0$, while
$h^0(\E(1))=1$, $h^0(\E(2))=7$ and $h^0(\E(3))=20$ give
$\dim\operatorname{Ext}^1(\E,\E)=2\cdot20-(3\cdot7+1)+1=19$.

The vanishing of the obstruction space $\operatorname{Ext}^2(\E,\E)$
then proves smoothness of the stable-sheaf moduli space at $[\E]$.

Finally, by Theorem~\ref{M3} the only irreducible components of
$\M(-1,3,7)$ are $\overline{\R(v_3)}$, of dimension $19$, and
$\operatorname{T}(-1,3,9,1)$, every closed point of which has Zariski tangent dimension at
least $23$. As $\dim\operatorname{Ext}^1(\E,\E)=19$,  $[\E]$ lies in
$\overline{\R(v_3)}$ and not in $\operatorname{T}(-1,3,9,1)$.  Additionally $[\E]$ is a smooth point
of $\overline{\R(v_3)}$. It is not reflexive, since $\E^{\vee\vee}/\E$ has
multiplicity $2$, and therefore lies in
$\overline{\R(v_3)}\setminus\R(v_3)$.
\end{proof}

\subsubsection*{The set-theoretic reflexive boundary}
\label{sec:boundary-decomposition}

We now give the complete set-theoretic description of the reflexive
boundary. Set
\[
T=\operatorname{T}(-1,3,9,1),
\]
where $T$ denotes the closed irreducible component. By Theorem~\ref{M3}, $\M(-1,3,7)=\overline{\R(v_3)}\cup T$, with
component dimensions $19$ and $23$.

For $(m,r)=(2,-2)$ or $(0,-3)$, let $ \operatorname{X}(-1,2,m,r,0)$ denote the
closure in $\M(-1,3,7)$ of the locus of kernels of epimorphisms
\[
0\longrightarrow\E\longrightarrow\F\longrightarrow\Oo_L(r)
\longrightarrow0,\qquad [\F]\in\R(-1,2,m),
\]
with $L\subset\Pt$ a line. Denote these defining kernel loci by
$\mathcal Y_2$ and $\mathcal Y_3$, respectively; thus
\[
 \operatorname{X}(-1,2,2,-2,0)=\overline{\mathcal Y_2},\qquad
 \operatorname{X}(-1,2,0,-3,0)=\overline{\mathcal Y_3}.
\]
The last parameter is zero because the defining quotient has no
zero-dimensional subsheaf. We retain all admissible lines, including
those meeting the singularities of $\F$, before taking the closures.
This extends the parameter notation of \cite[Theorem 10]{almeida2022irreducible} to the
exceptional twists $r=-2,-3$, which fall outside the hypotheses
$r\geq e$ and $m>0$ of Theorem~\ref{Xscheme}. Note that no statement of that theorem
is claimed for them, and the symbols above are defined directly by the
displayed extensions.

As before, $ \operatorname{X}(-1,2,4,-1,0)$ denotes the closure in
$\M(-1,3,7)$ of the kernels with $[\F]\in\R(-1,2,4)$,
$L\cap\operatorname{Sing}(\F)=\varnothing$, and
$\F|_L\simeq\Oo_L(-1)\oplus\Oo_L$.

\begin{proposition}\label{prop:boundary-v3}
With the notation above, we have
\begin{equation}\label{eq:boundary-v3}
\overline{\R(v_3)}\setminus\R(v_3) =
\bigl(\overline{\R(v_3)}\cap T\bigr)\cup \operatorname{X}(-1,2,4,-1,0)\cup\mathcal C(-1,1,1,2) \cup \mathcal Y_2 \cup \mathcal Y_3
\end{equation}
In addition, $ \operatorname{X}(-1,2,4,-1,0)$ is irreducible of dimension $15$, and its
generic point is a smooth point of $\M(v_3)$ lying outside $T$.

\end{proposition}

\begin{proof}
We start by proving that
\begin{equation}\label{eq:T-homological-dimension}
T=\{[\E]\in\M(v_3) \mid \operatorname{hd}(\E)=2\}.
\end{equation}
 Let $[\E]\in\M(v_3)$ be a non-reflexive
sheaf, set $\F=\E^{\vee\vee}$ and $\mathcal Q=\F/\E$, and let $\mathcal Z$
be the maximal zero-dimensional subsheaf of $\mathcal Q$. The local Ext
calculation that establishes \eqref{eq:length} gives
\[
\lext^2(\E,\Oo_{\Pt})\simeq\lext^3(\mathcal Z,\Oo_{\Pt}),
\]
so that $\operatorname{hd}(\E)=2$ whenever $\mathcal Z\neq0$, while
$\operatorname{hd}(\E)=1$ otherwise, since in this case the nonzero pure
quotient makes $\E$ non-locally-free along a curve.

Assume first that $\mathcal Z\neq0$. It follows from
Section~\ref{sec:spectra-v3} and \eqref{eq:length} that
$\mathcal Z\simeq\Oo_p$. Set
$\mathcal G\coloneqq\ker\{\F\to\mathcal Q/\mathcal Z\}$, so that
$\E\subset\mathcal G\subset\F$ and
\[
0\longrightarrow\E\longrightarrow\mathcal G\longrightarrow\Oo_p
\longrightarrow0 ;
\]
in particular $c(\mathcal G)=(-1,3,9)$. Note that $\mathcal G$ need not be
reflexive, since $\F/\mathcal G\simeq\mathcal Q/\mathcal Z$ may be nonzero.
It is, however, stable: any rank one subsheaf of $\mathcal G$ is a subsheaf
of the slope-stable sheaf $\F$, and $c_1(\mathcal G)=c_1(\F)=-1$. By
\cite[Theorem 3.1(i)(b) and Section 4.4]{schmidt2020rank}, $\mathcal G$ is
an extension of $\Oo_H(-3)$ by $\Oo_{\Pt}(-1)^{\oplus2}$, and hence admits
a resolution
\[
0\longrightarrow\Oo_{\Pt}(-4)\xrightarrow{\alpha}
\mathcal V\longrightarrow\mathcal G\longrightarrow0,
\qquad\mathcal V=\Oo_{\Pt}(-3)\oplus\Oo_{\Pt}(-1)^{\oplus2}.
\]
Write $\alpha=(f,g_1,g_2)$ according to the decomposition of $\mathcal V$,
so that $f$ is a linear form and $g_1,g_2$ are cubics, and set
$A\coloneqq\Hom(\Oo_{\Pt}(-4),\mathcal V)$.

The epimorphism $\mathcal G\to\Oo_p$ is given by a nonzero functional
$u\colon\mathcal V|_p\to\mathbb C$ satisfying $u\alpha(p)=0$. For fixed $p$
and $u$, the maps satisfying this condition form a linear subspace
$\Lambda\subset A$, which is therefore irreducible. Let $W\subset\Lambda$
be the subset of maps with stable torsion-free cokernel. It contains the
given $\alpha$, and it is open in $\Lambda$: the cokernel of a nonzero map
$\Oo_{\Pt}(-4)\to\mathcal V$ is torsion-free if and only if its three
entries have no common factor, which is an open condition, and stability is
open in flat families (the cokernels form a flat family over
$W$, see below). In particular $W$ is irreducible.

We now globalize the construction over $W\times\Pt$. Denote by $\pi_1$ and
$\pi_2$ the projections onto the factors. Since $W\subset A$, there is a
tautological morphism
\[
\Phi\colon\pi_2^*\Oo_{\Pt}(-4)\longrightarrow\pi_2^*\mathcal V
\]
whose restriction to $\{\alpha\}\times\Pt$ is $\alpha$ itself. Each such
restriction is injective, since a nonzero map from a line bundle to a
torsion-free sheaf is injective; hence $\Phi$ is injective and
$\mathbf G\coloneqq\coker\Phi$ is flat over $W$, with
$\mathbf G|_{\{\alpha\}\times\Pt}\simeq\coker(\alpha)$. Restriction to
$W\times\{p\}$ followed by $u$ gives a morphism
$\pi_2^*\mathcal V\to\Oo_{W\times\{p\}}$, which vanishes on the image of
$\Phi$ because $u\alpha(p)=0$ for every $\alpha\in W$. It therefore factors
through an epimorphism $\mathbf G\to\Oo_{W\times\{p\}}$, and we set
$\mathbf E\coloneqq\ker\{\mathbf G\to\Oo_{W\times\{p\}}\}$. In the exact
sequence
\[
0\longrightarrow\mathbf E\longrightarrow\mathbf G
\longrightarrow\Oo_{W\times\{p\}}\longrightarrow0
\]
the middle and right terms are flat over $W$, so $\mathbf E$ is flat as
well, and its restriction to $\{\alpha\}\times\Pt$ is the kernel $E_\alpha$
of the epimorphism $\coker(\alpha)\to\Oo_p$ induced by $u$. Every
$E_\alpha$ is stable with Chern classes $(-1,3,7)$, so $\mathbf E$ defines a
morphism
\[
W\longrightarrow\M(v_3),\qquad \alpha\longmapsto[E_\alpha],
\]
sending the given $\alpha$ to $[\E]$.

Finally, let $W'\subset\Lambda$ be the set of maps $\alpha=(f,g_1,g_2)$ such
that $f,g_1,g_2$ have finitely many common zeros and $\alpha(p)\neq0$. Both
are nonempty open conditions on the irreducible set $\Lambda$, so $W'$ is a
nonempty open subset of $\Lambda$. For $\alpha\in W'$ the cokernel
$\coker(\alpha)$ is reflexive, because its degeneracy locus is
zero-dimensional, and locally free at $p$; it is also stable, since its
resolution gives $H^0(\coker(\alpha))=0$ and its first Chern class equals
$-1$. Thus $W'\subset W$, hence $W'$ is dense in $W$, and for $\alpha\in W'$
the sheaf $[E_\alpha]$ lies in the defining locus of $T$
described in Theorem~\ref{0dcomp}, since $\coker(\alpha)$ is a
sheaf in $\R(-1,3,9)$ which is locally free at $p$. The image of $W$
in $\M(v_3)$ is therefore contained in the closure of the image of $W'$,
which lies in $T,$ since $T$ is closed, and we conclude that $[\E]\in T$.

Conversely, we show that
\[
\M(v_3)_{\leq1}\coloneqq\{[\E]\in\M(v_3)\mid\operatorname{hd}(\E)\leq1\}
\]
is open in $\M(v_3)$. Recall that $\M(v_3)$ is a good quotient $\pi\colon R^{ss}\to\M(v_3)$ of an
open subset $R^{ss}$ of a Quot scheme, carrying a universal family
$\mathbf U$ on $R^{ss}\times\Pt$. Since every semistable sheaf with
$c_1=-1$ is stable, two points of $R^{ss}$ have the same image under $\pi$
if and only if the corresponding sheaves are isomorphic, so
$\pi^{-1}(\M(v_3)_{\leq1})=\{q\in R^{ss}\mid\operatorname{hd}(\mathbf U_q)\leq1\}$.
As $\pi$ is submersive, $\M(v_3)_{\leq1}$ is open in $\M(v_3)$
if and only if this preimage is open in $R^{ss}$.
It is therefore enough to check that, for a flat
family $\mathbf E$ of sheaves on $\Pt$ parametrized by a scheme $S$, the set
\[
S_{\leq1}\coloneqq\{s\in S\mid\operatorname{hd}(\mathbf E_s)\leq1\}
\]
is open in $S$, where $\mathbf E_s\coloneqq\mathbf E|_{\{s\}\times\Pt}$.

Since openness is a local property, we may assume that $S$ is affine.
Denoting again by $\pi_1$ and $\pi_2$ the projections of
$S\times\Pt$ onto its factors, by
the relative version of Serre's theorem,
$\pi_1^*\pi_{1*}(\mathbf E(m))\to\mathbf E(m)$ is surjective for $m\gg0$,
and $\pi_{1*}(\mathbf E(m))$, being coherent on the affine scheme $S$, is
generated by finitely many global sections, say $N$ of them. Composing, we
obtain an epimorphism $\mathbf L\to\mathbf E$ with
$\mathbf L=\pi_2^*\Oo_{\Pt}(-m)^{\oplus N}$. Let $\mathbf K$ be its kernel.
Since $\mathbf E$ is flat over $S$, so is $\mathbf K$, being the kernel of
an epimorphism between flat sheaves, and for every $s\in S$ the sequence
\[
0\longrightarrow\mathbf K_s\longrightarrow\mathbf L_s
\longrightarrow\mathbf E_s\longrightarrow0
\]
is exact. Since $\mathbf L_s$ is locally free,
$\operatorname{hd}(\mathbf E_s)\leq1$ if and only if $\mathbf K_s$ is
locally free; and since $\mathbf K$ is flat over $S$, this happens if and
only if $\mathbf K$ is locally free along $\{s\}\times\Pt$. The locus where
$\mathbf K$ fails to be locally free is closed in $S\times\Pt$, hence its
image under $\pi_1$ is closed in $S$ because $\Pt$ is proper, and the
complement of this image is exactly $S_{\leq1}$. Thus $S_{\leq1}$ is open,
as claimed.

Now let $T^\circ\subset T$ denote the defining locus of $T$, namely the
set of points $[\E]\in\M(v_3)$ such that $\E$ is the kernel of an
epimorphism $\F\to\Oo_p$ with $[\F]\in\R(-1,3,9)$ locally free at $p$; by
Theorem~\ref{0dcomp}, $T$ is the closure of $T^\circ$. For such $\E$,
applying $\lhom(-,\Oo_{\Pt})$ to the sequence
$0\to\E\to\F\to\Oo_p\to0$ and using that
$\lext^i(\F,\Oo_{\Pt})=0$ for $i\geq2$, since $\F$ is reflexive, we obtain
\[
\lext^2(\E,\Oo_{\Pt})\simeq\lext^3(\Oo_p,\Oo_{\Pt})\simeq\Oo_p\neq0 ,
\]
so that $\operatorname{hd}(\E)=2$ and $T^\circ\cap\M(v_3)_{\leq1}=\varnothing$.
On the other hand, $T\cap\M(v_3)_{\leq1}$ is an open subset of the
irreducible scheme $T$, and $T^\circ$ is dense in $T$, so
$T\cap\M(v_3)_{\leq1}$ would meet $T^\circ$ if it were nonempty. Hence
every $[\E]\in T$ satisfies $\operatorname{hd}(\E)=2$, which together with
the first part proves \eqref{eq:T-homological-dimension}. In particular
$\R(v_3)\cap T=\varnothing$, since a reflexive sheaf on $\Pt$ has
homological dimension at most $1$.

Combining \eqref{eq:T-homological-dimension} with Theorem~\ref{M3}, every
non-reflexive $[\E]\in\M(v_3)$ with $\operatorname{hd}(\E)=1$ lies in
$\overline{\R(v_3)}\setminus(\R(v_3)\cup T)$. The loci $\mathcal Y_2$,
$\mathcal Y_3$ and $\mathcal C(-1,1,1,2)$ consist of sheaves whose quotient
$\E^{\vee\vee}/\E$ is pure of dimension one, hence of homological dimension
$1$ by the first paragraph of the proof, so they lie in $\overline{\R(v_3)}\setminus\R(v_3)$. 

Before proving \eqref{eq:boundary-v3}, we check that every kernel of an
epimorphism $\F\to\Oo_L(-1)$ with $[\F]\in\R(-1,2,4)$ belongs to
$ \operatorname{X}(-1,2,4,-1,0)$, including those for which $L$ meets $\sing(\F)$;
recall that $ \operatorname{X}(-1,2,4,-1,0)$ was defined as the closure of the
kernels for which $L\cap\sing(\F)=\varnothing$. By
\cite[Lemma 9.6]{hartshorne1980reflexive}, every $[\F]\in\R(-1,2,4)$ has a
resolution
\[
0\longrightarrow\Oo_{\Pt}(-3)
\xrightarrow{\alpha=(\ell,q_1,q_2)}
\Oo_{\Pt}(-2)\oplus\Oo_{\Pt}(-1)^{\oplus2}
\longrightarrow\F\longrightarrow0,
\]
where $\ell$ is a linear form and $q_1,q_2$ are quadrics. Since
$\Hom(-,\Oo_L(-1))$ is left exact, an epimorphism $\F\to\Oo_L(-1)$ is the
same as a triple
\[
\beta=(h,a,b),\qquad h\in H^0(\Oo_L(1)),\quad (a,b)\in\mathbb C^2,
\]
satisfying $h\,\ell|_L+a\,q_1|_L+b\,q_2|_L=0$ in $H^0(\Oo_L(2))$ and
inducing a surjection $\Oo_L(-2)\oplus\Oo_L(-1)^{\oplus2}\to\Oo_L(-1)$;
the latter holds if and only if $(a,b)\neq(0,0)$, since a section of
$\Oo_L(1)$ has a zero. The pairs $(L,\beta)$ with $(a,b)\neq(0,0)$ form an
irreducible variety $B$, and for each of them the relation above is the
kernel of a surjective linear map from
$H^0(\Oo_{\Pt}(1))\oplus H^0(\Oo_{\Pt}(2))^{\oplus2}$ to $H^0(\Oo_L(2))$,
so that the admissible $\alpha$'s form a vector bundle
$\mathcal A\to B$. Let $\mathcal A^\circ\subset\mathcal A$ be the open
subset of pairs $((L,\beta),\alpha)$ for which $\coker(\alpha)$ is stable
and reflexive, and let $\mathcal A^{\circ\circ}\subset\mathcal A^\circ$ be
the open subset where moreover $L\cap\sing(\F)=\varnothing$. Both are
nonempty: for a generic $[\F]\in\R(-1,2,4)$ and a generic line $L$ we have
$L\cap\sing(\F)=\varnothing$, since $\sing(\F)$ is finite, and
$\F|_L\simeq\Oo_L(-1)\oplus\Oo_L$ by the Grauert--M\"ulich theorem, so
that projection onto the first summand gives an epimorphism
$\F\to\Oo_L(-1)$. Hence $\mathcal A^\circ$ is irreducible and
$\mathcal A^{\circ\circ}$ is dense in it.

The kernels $\ker\{\beta\colon\coker(\alpha)\to\Oo_L(-1)\}$ form a flat
family over $\mathcal A^\circ$, giving a morphism
$\mathcal A^\circ\to\M(v_3)$ whose image is precisely the set of kernels we
want to locate. For a pair in $\mathcal A^{\circ\circ}$, the quotient gives
an exact sequence
\[
0\longrightarrow\Oo_L\longrightarrow\F|_L
\longrightarrow\Oo_L(-1)\longrightarrow0
\]
on $L$, which splits because $H^1(\Oo_L(1))=0$, so the corresponding kernel
lies in the defining locus of $ \operatorname{X}(-1,2,4,-1,0)$. The image of
$\mathcal A^\circ$ is therefore contained in the closure of the image of
$\mathcal A^{\circ\circ}$, that is, in $ \operatorname{X}(-1,2,4,-1,0)$, as claimed.
Finally, every such kernel has homological dimension $1$, so it lies in
$\overline{\R(v_3)}\setminus(\R(v_3)\cup T)$ by
\eqref{eq:T-homological-dimension}, and its closure lies in
$\overline{\R(v_3)}\setminus\R(v_3)$, the non-reflexive locus being closed.

We can now prove \eqref{eq:boundary-v3}. Let
$[\E]\in\overline{\R(v_3)}\setminus\R(v_3)$, and let $\mathcal Q$ and
$\mathcal Z$ be as above. If $\mathcal Z\neq0$, then
$\operatorname{hd}(\E)=2$, so $[\E]\in\overline{\R(v_3)}\cap T$ by
\eqref{eq:T-homological-dimension}. Otherwise $\mathcal Q$ is pure of
dimension one; write $d=\operatorname{mult}(\mathcal Q)$. By
\cite[Proposition 14(ii) and (iv)]{almeida2022irreducible},
the hull $\F$
is stable with $c_2(\F)=3-d$, and $c_2(\F)\geq1$ by
\cite[Corollary 3.3]{hartshorne1980reflexive}, so that $d=1$ or $d=2$. If $d=2$, then
$[\E]\in\mathcal C(-1,1,1,2)$ by Definition~\ref{def:mult-two-locus}.

In the case $d=1$, the support of $\mathcal Q$ is a line $L$, and purity
implies that $\mathcal Q\simeq\Oo_L(r)$ for some integer $r$. Indeed, since
$\mathcal Q$ has multiplicity $1$, it is annihilated by $\mathcal I_L$ at
the generic point of $L$, so $\mathcal I_L\cdot\mathcal Q$ is a subsheaf of
$\mathcal Q$ with zero-dimensional support, hence zero. Thus $\mathcal Q$ is
a torsion-free $\Oo_L$-module of rank one on $L\simeq\p1$, that is, a line
bundle. Here $c_2(\F)=2$, and \cite[Corollary 2.4 and Theorem
8.2]{hartshorne1980reflexive} give $m\coloneqq c_3(\F)\in\{0,2,4\}$.
Subtracting the Riemann--Roch polynomials in \eqref{eq:RR-spectrum} we
obtain
\[
P_{\mathcal Q}(t)=P_\F(t)-P_\E(t)=t+\frac{m-4}{2},
\]
and comparing with $P_{\Oo_L(r)}(t)=t+r+1$ gives $r=(m-6)/2$. 
Thus $(m,r)=(4,-1)$, $(2,-2)$ or $(0,-3)$. In the first case
$[\E]\in \operatorname{X}(-1,2,4,-1,0)$, as shown above, while in the other two
$[\E]$ lies in $\mathcal Y_2$ or $\mathcal Y_3$. This proves that the
left hand side of \eqref{eq:boundary-v3} is contained in the right hand
side. Conversely, each term of the right hand side is contained in
$\overline{\R(v_3)}\setminus\R(v_3)$: for $\overline{\R(v_3)}\cap T$ this
follows from $\R(v_3)\cap T=\varnothing$, and for $\operatorname{X}(-1,2,4,-1,0)$, $\mathcal Y_2$, $\mathcal Y_3$ and $\mathcal C(-1,1,1,2)$ it was proved above. This gives
\eqref{eq:boundary-v3}.

It remains to prove the last assertion. Every $[\F]\in\R(-1,2,4)$ satisfies
$\Ext^2(\F,\F)=0$, as follows from the resolution displayed above, and
$\R(-1,2,4)$ is irreducible by \cite[Theorem 9.2]{hartshorne1980reflexive}.
Theorem~\ref{Xscheme} with $(e,n,m,r,s)=(-1,2,4,-1,0)$ then gives that
$\operatorname{X}(-1,2,4,-1,0)$ is irreducible of dimension $15$, and that
$\dim\Ext^1(\E,\E)=19$ for its generic point $[\E]$. Such a point lies in
the defining locus, so $\operatorname{hd}(\E)=1$ and $[\E]\notin T$ by
\eqref{eq:T-homological-dimension}; on the other hand
$[\E]\in\overline{\R(v_3)}$ by \eqref{eq:boundary-v3}. Since
$\overline{\R(v_3)}$ and $T$ are the only irreducible components of
$\M(v_3)$, the dimension of $\M(v_3)$ at $[\E]$ is
$\dim\overline{\R(v_3)}=19$, which coincides with the dimension of the
tangent space. Hence $[\E]$ is a smooth point of $\M(v_3)$, as desired.
\end{proof}

We emphasize that the union on display \eqref{eq:boundary-v3} is not claimed to be disjoint or to be a decomposition into irreducible components. It would be interesting to further study the loci $\mathcal Y_2$, $\mathcal Y_3$ and $\mathcal C(-1,1,1,2)$.

\bibliography{mref}
\bibliographystyle{alpha}

\end{document}